\documentclass{amsart}
\usepackage{amssymb}
\usepackage{graphicx,epsf}
\usepackage{color}
\usepackage{mathrsfs}
\usepackage{latexsym}
\usepackage{verbatim}
\usepackage{amsmath}
\usepackage{pstricks}
\usepackage{pst-plot}

 \newtheorem{thm}{Theorem}[section]
 \newtheorem{lem}[thm]{Lemma}
 
 \newtheorem{prop}[thm]{Proposition}
 \newtheorem*{ThmA}{Theorem}
 \newtheorem*{CorA}{Corollary}

 \theoremstyle{definition}

 \theoremstyle{remark}
 \newtheorem{rem}[thm]{Remark}
 \newtheorem{remNot}{Remark on the notations}
 \numberwithin{equation}{section}

\def\polhk#1{\setbox0=\hbox{#1}{\ooalign{\hidewidth
    \lower1.5ex\hbox{`}\hidewidth\crcr\unhbox0}}}

\renewcommand{\Im}{\operatorname{Im}} 

\newcommand{\Dom}{\operatorname{Dom}\,}
\newcommand{\Dil}{\operatorname{Dil}\,}
\newcommand{\daron}{\operatorname{int}}
\renewcommand{\mod}{\operatorname{mod}\,}
\newcommand{\eps}{\varepsilon}
\newcommand{\pc}{\mathcal{PC}}
\newcommand{\alf}{\alpha}
\newcommand{\cc}{\mathbb{C}}
\newcommand{\ra}{\rightarrow}
\newcommand{\rr}{\mathcal{R}}
\newcommand{\M}{\mathcal{M}} 
\newcommand{\B}{\mathcal{B}} 
\newcommand{\mfr}{\mathfrak{R}}
\newcommand{\hh}{\mathscr{H}}
\newcommand{\LL}{\mathscr{L}}
\newcommand{\BV}{\mathbf{V}}

\newcommand{\BU}{\mathbf{U}}
\newcommand{\Bpsi}{\mbox{\boldmath $\psi$}}
\newcommand{\BDelta}{\mbox{\boldmath $\Delta$}}
\newcommand{\Bh}{\mathbf{h}}
\newcommand{\Bf}{\mathbf{f}}
\newcommand{\BJ}{\mathbf{J}}
\newcommand{\BB}{\mathbf{B}}
\newcommand{\BL}{\mathbf{L}}
\newcommand{\Bi}{\textbf{i}}
\newcommand{\BE}{\textbf{E}}
\newcommand{\BK}{\textbf{K}}

\begin{document}
\title[Combinatorial Rigidity of unicritical polynomials]{Combinatorial Rigidity for Some Infinitely Renormalizable 
Unicritical Polynomials}
\author[D. Cheraghi]{Davoud Cheraghi}
\address {Mathematics Department\\  Stony Brook University\\ Stony Brook,  NY 11794, USA}
\email{d.cheraghi@warwick.ac.uk}
\date{}

\begin{abstract}
Here we prove that \textit{infinitely renormalizable} unicritical polynomials $P_c:z \mapsto z^d+c$, 
with $c\in\mathbb{C}$, satisfying \textit{a priori bounds} and a certain ``combinatorial'' condition 
are \textit{combinatorially rigid}. 
This implies the local connectivity of the connectedness loci (the Mandelbrot set when $d=2$) at the
corresponding parameters.
\end{abstract}

\maketitle

\section{Introduction}
The \textit{Multibrot set} $\M_d$, or the \textit{connectedness locus} of the unicritical polynomials, 
is the set of parameter values $c$ in $\mathbb{C}$ for which the
\textit{Julia set} of $P_c:z \mapsto z^d+c$ is connected. 
The set $\M_2$ is the well-known \textit{Mandelbrot} set.

There is a way of defining graded partitions (puzzle pieces) of the Multibrot set 
such that the dynamics of the maps $P_c$ in each piece have some
special combinatorial property. All maps in a given piece of a partition
of a certain level are called \textit{combinatorially
equivalent} up to that level. Conjecturally, combinatorially
equivalent (up to all levels) \textit{non-hyperbolic} maps in this family are
\textit{conformally conjugate}. As stated in \cite{DH1a} for $d=2$,
this \textit{rigidity conjecture} is equivalent to the local
connectivity of the Mandelbrot set, and it naturally extends to
degree $d$ unicritical polynomials. In the quadratic case, this
conjecture is formulated as MLC by Douady and Hubbard. They
also proved there that MLC implies the density of \textit{hyperbolic polynomials} in the space of
quadratic polynomials. These discussions have been extended to
degree $d$ unicritical polynomials by Schleicher in \cite{Sch}.

In the 1990's, Yoccoz proved that $\M_2$ is locally connected at all non-hyperbolic parameter values
which are at most finitely \textit{renormalizable}. He also proved 
the local connectivity of the Julia sets of these maps with all 
periodic points repelling (see \cite{H}). The degree two assumption was essential in his proof.

In \cite{Ly97},  Lyubich proved the combinatorial rigidity conjecture for a class
of infinitely renormalizable quadratic polynomials. These are quadratic 
polynomials satisfying a \textit{secondary limbs condition}, denoted by 
$\mathcal{SL}$,  with sufficiently high return times. The proof in this case
also relies on the degree two assumption. 

The local connectivity of the Julia sets of degree $d$
unicritical polynomials which are at most finitely
renormalizable and with all periodic points repelling has been shown in \cite{KL1}. Their proof 
is based on ``controlling'' the geometry of a \textit{modified principal nest}. The same
controlling technique has been used to settle the rigidity problem for
these parameters in \cite{AKLS}. 

Recently, the \textit{a priori} bounds property, a type of compactness on renormalization levels, has been established for more parameters. In \cite{K2}, it is proved for 
infinitely \textit{primitively} renormalizable maps of bounded type. In \cite{KL2}, it is
proved for parameters satisfying a \textit{decorations} condition and in \cite{KL3},
under a \textit{molecule} condition. Here we prove that the \textit{a priori} bounds property, under 
the $\mathcal{SL}$ condition, implies the combinatorial rigidity conjecture for infinitely renormalizable maps. 
The $\mathcal{SL}$ class includes all parameters for which \textit{a priori} bounds is known to us.

\begin{ThmA}[Rigidity]
Let $P_c$ be an infinitely renormalizable degree $d$ unicritical polynomial satisfying the \textit{a priori bounds}
and $\mathcal{SL}$ conditions. Then $P_c$ is combinatorially rigid.
\end{ThmA}
 
This result was proved in part II of \cite{Ly97} for quadratic polynomials. 
That  proof as well as the one presented here are based on the Sullivan-Thurston pullback argument. 
However, the one in \cite{Ly97} uses linear growth of certain moduli along the principal nest which does not 
hold for arbitrary degree unicritical polynomials. 
It turns out that the definite modulus of certain annuli in a modified principal nest introduced in \cite{AKLS} helps us 
to easily ``pass'' over the principal nest. 
This makes the whole construction simpler and more general so that we can  include unicritical polynomials of arbitrary degree. Combining the above theorem with \cite{KL2} and \cite{KL3} 
we obtain the following:
{\samepage
\begin{CorA}
Assume that $P$ and $\tilde{P}$ are combinatorially equivalent infinitely renormalizable unicritical polynomials 
with one of the following conditions:
\begin{itemize}
\item[--] $P$ and $\tilde{P}$ are quadratic and satisfy the molecule condition, 

\hspace{-.5cm}or,
 
\item[--] $P$ and $\tilde{P}$ have arbitrary degree and satisfy the decoration condition.
\end{itemize}
Then, $P$ and $\tilde{P}$ are conformally equivalent.
\end{CorA}
}
The rigidity problem for a separate combinatorial class of quadratics is treated by a wholly
different approach in \cite{L09} which does not involve the \textit{a priori} bounds property.

The rigidity conjecture for real polynomials has been fully established over the last twenty years. 
The quadratic case was accomplished, independently, in \cite{Ly97} and \cite{GS98}. 
The real multi-critical case was treated in \cite{LvS98}. 
One may refer to these for further references. Our result can be applied to real unicritical polynomials as well.
Therefore, combining with \cite{S1}, it gives a new proof of the density of hyperbolicity in the family 
$x\ra x^{2k}+c$, $k=1,2,\dots$, which was proved earlier in \cite{KSvS07}.

The structure of the paper is as follows. In $\S 2$ we introduce the basics of holomorphic 
dynamics required for our work. In $\S 3$, Yoccoz puzzle pieces are defined, the modified principal
nest is introduced, and combinatorics of unicritical polynomials is discussed. The proof of the main 
theorem, presented in Section $4$, is reduced to the existence of a \textit{Thurston conjugacy} by the pullback 
method. To build such a conjugacy, we start with a topological conjugacy on the whole complex 
plane and then step by step, on finer and finer scales, replace this homeomorphism by 
quasi-conformal maps while sacrificing the equivariance property but staying in the ``right'' homotopy class. 
At the end one obtains a quasi conformal map on the complex plane homotopic to a topological conjugacy 
relative the post-critical set, that is, a Thurston conjugacy.
\subsection*{Acknowledgment}
I am indebted to M. Lyubich for suggesting the problem; this paper would have never been
finished without his great patience. Further thanks are due to R. P\'{e}rez for our very useful
discussions on the combinatorics of the Mandelbrot set. 
Finally, many heartfelt thanks to the referee for many valuable comments to improve the writing of this work.

\section{Polynomials and the connectedness loci}
\subsection{External rays and Equipotentials}
One can read more about the following basics of holomorphic dynamics in \cite{Mi06} and \cite{Br94}.

Let $f:\mathbb{C} \rightarrow \mathbb{C}$ be a monic polynomial of
degree $d$: $f(z)=z^{d}+a_{1}z^{d-1}+\cdots+a_{d}$. Infinity is a
super attracting fixed point of $f$ whose \textit{basin of attraction} is defined as
\[D_{f}(\infty):=\{z\in\mathbb{C}:f^{n}(z):=\overbrace{f\circ f\circ\dots\circ f}^{n \text{ times}}(z)\ra\infty\text{ as }n\ra\infty\}.\]
The complement of $D_{f}(\infty)$ is called the \textit{filled Julia} set: $
K(f):=\mathbb{C}\setminus D_{f}(\infty)$. The \textit{Julia set},
$J(f)$, is defined as the boundary of $K(f)$. 
It is well-known that the Julia set and the filled Julia set of a polynomial are connected
if and only if the orbit of all critical points stay bounded under iteration.

With $f$ as above,  there exists a conformal change of coordinate $B_{f}$,
the \textit{B\"ottcher coordinate}, which conjugates $f$ to the
$d$-th power map $z \mapsto z^{d}$ throughout some neighborhood of
infinity $U_f$. That is,
\begin{equation}
B_{f}:U_{f}\rightarrow \{z\in \mathbb{C}:|z|>r_{f}\geq1\} \label{Bottcher}
\end{equation}
with $B_{f}(f(z))=(B_{f}(z))^{d}$, and $B_{f}(z)\sim z$ as $z\rightarrow \infty$.

In particular, if the filled Julia set is connected, $B_{f}$
coincides with the Riemann mapping of $D_f(\infty)$ onto the complement of the closed unit disk normalized 
to be tangent to the identity map at infinity.

The \textit{external ray} (or \textit{ray} for short) \textit{of angle}
$\theta$ is defined as 
\[R^{\theta}=R^{ \theta }_{f}:=B_f^{-1} \{re^{i\theta} : r_{f} <r<\infty\}.\]
The \textit{equipotential of level} $r > r_f$ is defined as
\[E^r=E^{r}_{f}:=B_f^{-1} \{re^{i\theta}: 0\leq \theta \leq 2\pi \}.\]
The \textit{equivariance} property of the map $B_f$, i.e.\ $B_{f}(f(z))=(B_{f}(z))^{d}$, 
implies that $f(R^{\theta})=R^{d\theta}$, and \text{$f(E^r)=E^{r^d}.$}

A ray $R^{\theta}$ is called \textit{periodic} of period $p$ if
$f^{p}(R^{\theta})=R^{\theta}$. 
A ray is fixed (has period $1$) if and only if $\theta$ is a rational number of the
form $2\pi j/(d-1)$. By definition, a ray $R^{\theta}$ \textit{lands} at a
well defined point $z$ in $J(f)$, if the limiting value of the ray
$R^{\theta}$ (as $r \ra 1$) exists and is equal to $z$. Such a point $z$ in $J(f)$ is
called the \textit{landing point} of the ray $R^\theta$. The following theorem
characterizes the landing points of the periodic rays. See
\cite{DH1a} for further discussions.

\begin{thm}\label{rays}
Let $f$ be a polynomial of degree $d\geq 2$ with connected Julia
set. Every periodic ray lands at a well defined periodic point which
is either repelling or parabolic. Vice versa,  every repelling or
parabolic periodic point is the landing point of at least one, and at
most finitely many periodic rays with the same ray period.
\end{thm}

In particular, this theorem implies that the external rays landing
at a periodic point are organized in several cycles. Suppose
$\overline{a}=\{a_{k}\}_{k=0}^{p-1}$ is a repelling or parabolic
cycle of $f$. Let $\mfr(a_{k})$ denote the union of all the external
rays landing at $a_{k}$. The configuration
\[\mfr (\overline {a})=\bigcup _{k=0}^{p-1} \mfr(a_{k}),\]
with the rays labeled by their external angles,  is called the
\textit{periodic point portrait} of $f$ associated to the cycle
$\overline{a}$.

\subsection{Unicritical family and the connectedness locus}
Any degree $d$ polynomial with only one critical point is conformally
conjugate to some $P_{c}(z)=z^{d}+c$, with $c\in \cc$. A case of
special interest is the following fixed point portrait.
The $d-1$ fixed rays $R^{2\pi j/(d-1)}$ land at $d-1$ (distinct) fixed points
called $\beta_{j}$. Moreover, these are the only rays that land at
$\beta_{j}$'s. 
Therefore, these fixed points are \textit{non-dividing}, that is, $K(P_c)\setminus \beta_{j}$ is 
connected for any $j$. 
If the other fixed point called $\alpha$ is also repelling, there are
at least two rays that land at it. 
Thus, the $\alpha$-fixed point is \textit{dividing} and by 
Theorem \ref{rays}, the rays landing at $\alf$ are permuted under the dynamics. The following
statement has been shown in \cite{Mi2} for quadratic polynomials.
The same ideas apply to prove it for degree $d$ unicritical polynomials.

\begin{prop} \label{sectors}
If at least two rays land at the $\alpha$ fixed point of $P_c$,
we have: 
\begin{itemize}
\item[--] the component of $\mathbb{C} \setminus P_c^{-1}(\mfr(\alpha))$ containing
the critical value is a sector bounded by two external rays;
\item[--] the component of $\mathbb{C}\setminus P_c^{-1}(\mfr(\alpha))$ containing the
critical point is a region bounded by $2d$ external rays landing in
pairs at the points $e^{2\pi j/d} \alpha$, for $j=0,1,\ldots,d-1$.
\end{itemize}
\end{prop}

The \textit{connectedness locus} $\M_d$, or the \textit{Multibrot} set of degree $d$, is defined as the
set of parameters $c$ in $\mathbb{C}$ for which $J(P_{c})$ is
connected. In particular, $\M_2$ is the
well-known  \textit{Mandelbrot} set. See Figures~\ref{mandelpic} and \ref{multi}.
\begin{figure}[ht]
\begin{center}
\includegraphics[scale=1.2]{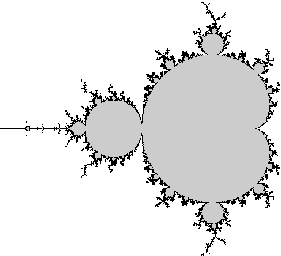}
\caption{The Mandelbrot set. The gray regions show the interior of $\M_2$ and 
         darker points show its boundary.}
\label{mandelpic}
\end{center}
\end{figure}

A well-known result due to Douady and Hubbard \cite{DH1a} shows
that these connectedness loci are connected. Their argument is based on 
considering the explicit conformal isomorphism
\[\B_d:\mathbb{C}\setminus \M_d \rightarrow \{z\in \mathbb{C}:| z| >1 \}\] 
given by $\B_d(c):=B_{c}(c)$, where $B_{c}$ is the
B\"{o}ttcher coordinate~\eqref{Bottcher} of $P_c$.

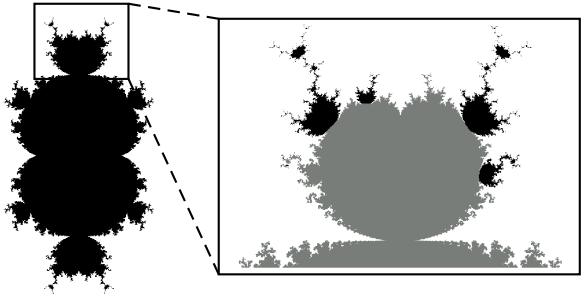
\begin{figure}[ht]
  \begin{center}
\begin{pspicture}(8,4)
\epsfxsize=2.5cm
\rput(1.1,2){\epsfbox{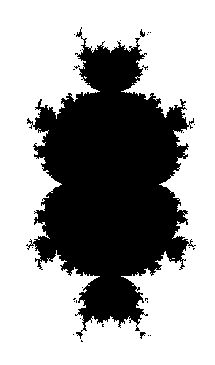}}
\epsfxsize=5cm
\rput(5.5,2){\epsfbox{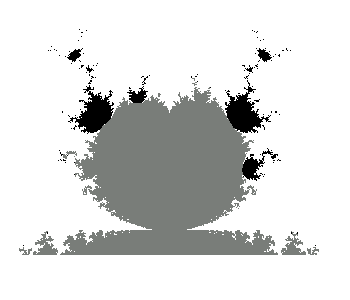}}
\pspolygon(.55,3)(1.8,3)(1.8,4)(.55,4)
\pspolygon(3,.4)(7.8,.4)(7.8,3.8)(3,3.8)
\psline[linestyle=dashed](1.8,3)(3,.4)
\psline[linestyle=dashed](1.8,4)(3,3.8)
%\psgrid[linecolor=gray]
\end{pspicture}
 \caption{Figure on the left shows the connectedness locus $\M_3$. The figure
             on the right is an enlargement of a primary limb in $\M_3$. The
             dark regions in the right box show some of the secondary limbs.}
    \label{multi}
  \end{center}
\end{figure}

By means of the conformal isomorphism $\B_d$, \textit{parameter
external rays} $R_{\theta}$ and \textit{equipotentials} $E_r$ are defined as 
$\B_d$-preimages of straight rays going to infinity and round circles around $0$,
respectively.

A polynomial $P_{c}$ (also the corresponding parameter $c$) is called
\textit{hyperbolic} if $P_{c}$ has an attracting periodic point. This attracting 
periodic point necessarily attracts orbit of the finite critical point. The
set of hyperbolic parameters in $\M_d$, which is open by definition, is a union of some
components of $\daron \M_d$. These components are called the \textit{hyperbolic components}.

\textit{The main hyperbolic component} is defined as the set of
parameter values $c$ for which $P_{c}$ has an attracting fixed
point. Outside of the closure of this set all fixed points
become repelling. Consider a parameter $c$ in a hyperbolic component $\hh \subset
\daron \M_d$, and suppose that $\overline {b}_{c}$ denotes the corresponding
attracting cycle with period $k>1$. On the boundary of $\hh$ this cycle
becomes neutral, and there are $d-1$ parameters $c_i\in \partial \hh$ where $P_{c_i}$
has a \textit{parabolic} cycle with multiplier equal to one. 
One of these parameters, which is called the \textit{root} of $\hh$ and is denoted by $c_{root}$, 
divides the connectedness locus into two pieces. 
Indeed, any hyperbolic component has
one root and $d-2$ \textit{co-roots}. The root is the landing point
of two parameter rays, while every co-root is the landing point of a
single parameter ray. See Figure~\ref{julia}. 
For a proof of these statements one may consult \cite{DH1a}, for the quadratic polynomials, 
and \cite{Sch} for arbitrary degree unicritical polynomials. 

If $c$ belongs to a hyperbolic component $\hh$ which is not the main hyperbolic 
component of the connectedness locus, the \textit{basin of attraction} of its attracting cycle $\bar{b}_c$, 
denoted by $A_c$, is defined as the set of points $z\in\mathbb{C}$ with 
$\langle P_c^n(z)\rangle_{n=0}^{\infty}$ converges to the cycle $\overline{b}_c$. 
The boundary of the component of $A_c$ containing $c$ is a Jordan curve which we denote it by $D_c$. 
The map $P^k_{c}$ on $D_{c}$ is topologically conjugate to $\theta \mapsto d \theta$ on the unit
circle. 
Therefore, there are $d-1$ fixed points of $P^k_c$ on this Jordan curve which are
repelling periodic points of $P_c$ of period dividing the period of
$\bar{b}_{c}$ (its period can be strictly less than the period of $\bar{b}_c$).
Among all rays landing at these repelling periodic points, let $\theta
_{1}$ and $\theta_{2}$ be the angles of the external rays bounding
the sector containing the critical value of $P_{c}$ (See Figure~\ref{julia}). The following
theorem makes a connection between external rays $R^{\theta_1}$,
$R^{\theta_2}$ and the parameter external rays
$R_{\theta_1}$, $R_{\theta_2}$. 
See \cite{DH1a} and \cite{Sch} for further details.

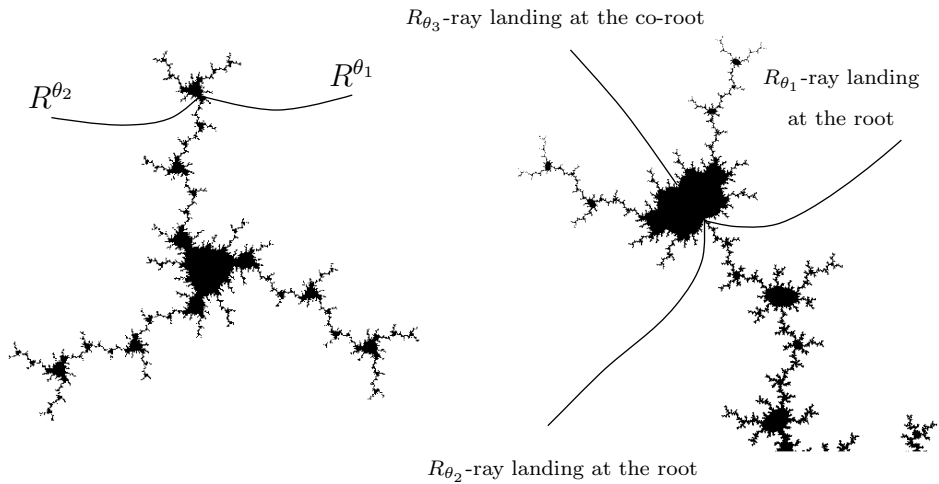
\begin{figure}[ht]
\begin{center}
\begin{pspicture}(12.4,8)
  \psset{xunit=1cm}
  \psset{yunit=1cm} 
%  \psframe(0,0)(8.5,7.5)
% \epsfxsize=5.5cm
   \rput(9.5,3.3){\epsfbox{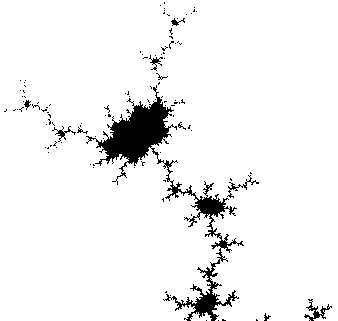}} 
   \epsfxsize=6cm
   \rput(2.6,3){\epsfbox{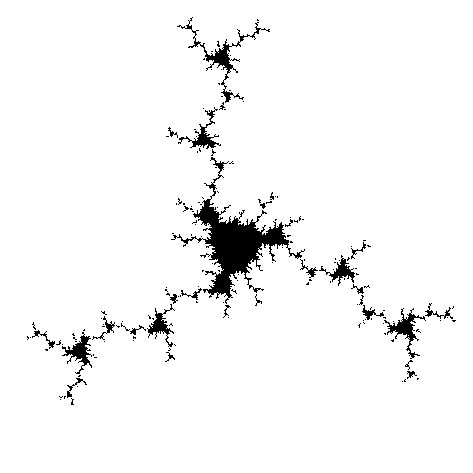}}
   
   \pscurve[linewidth=.5pt,showpoints=false](4.5,5.2)(3.5,5)(2.48,5.19)
   \pscurve[linewidth=.5pt,showpoints=false](.5,4.9)(1.5,4.8)(2.1,4.9)(2.48,5.19) 
   \rput(4.5,5.5){$R^{\theta_1}$}
   \rput(.5,5.2){$R^{\theta_2}$}

   \pscurve[linewidth=.5pt,showpoints=false](7.4,5.8)(8.1,5)(8.82,4)   
   \pscurve[linewidth=.5pt,showpoints=false](7.1,.8)(7.3,1.)(7.8,1.5)(8.8,2.4)(9.15,3)(9.18,3.53)
   \pscurve[linewidth=.5pt,showpoints=false](11.8,4.6)(10.2,3.5)(9.6,3.45)(9.18,3.53)

  \rput(7.2,6.2){\tiny{$R_{\theta_3}$-ray landing at the co-root}}
  \rput(7.3,.2){\tiny{$R_{\theta_2}$-ray landing at the root}}
  \rput(11,5.4){\tiny{$R_{\theta_1}$-ray landing}}
  \rput(11,4.9){\tiny{at the root}}  
\end{pspicture}
\caption{The figure on the left shows a primitively renormalizable
             Julia set, as well as the external rays $R^{\theta_1}$ and
             $R^{\theta_2}$ landing at the corresponding repelling periodic point. 
             The figure on the right is the corresponding primitive little
             Multibrot copy. It also shows the parameter external rays
             $R_{\theta_1}$ and  $R_{\theta_2}$ landing at the root point.}
\label{julia}
\end{center}
\end{figure}

\begin{thm}
The parameter external rays $R_{\theta_1}$ and $R_{\theta _2}$ land
at the root of $\hh$. Moreover, these are the only rays that land at this point.
\end{thm}

The closure of $R_{\theta _{1}}$ and $R_{\theta _{ 2}}$
cut the parameter plane into two components. The one containing $\hh$ with the root point 
attached to it is called the \textit{wake} $W_{\hh}$. So a wake is an open set with a point
attached to its boundary. 
Given a wake $W_\hh$ and an equipotential $E_\eta$, the \textit{truncated wake}
$W_\hh(\eta)$ is the bounded component of $W_\hh \setminus
E_\eta$. The part of the connectedness locus contained in $W_\hh$
is called the \textit{limb} $\LL_\hh$ of the connectedness locus \textit{originating} at $\hh$. 
In other words, $\LL_\hh=W_\hh\cap \M_d$. By definition, every limb is a closed set.

The wakes attached to the main hyperbolic component of $\M_d$ are called \textit{primary
wakes}. A limb associated to such a primary wake is called a 
\textit{primary limb}. If $\hh$ is a hyperbolic component attached to
the main hyperbolic component, all the wakes attached to $\hh$ (except $W_\hh$ itself)
are called \textit{secondary wakes}. Similarly, the limb associated
to a secondary wake is called a \textit{secondary limb}. A
\textit{truncated limb} is obtained from a limb by removing a
neighborhood of its root. Some secondary limbs are shown in Figure~\ref{multi}.

Given a parameter $c$ in a hyperbolic component $\hh$, we have the attracting cycle $\overline {b}_{c}$
as above, and the associated repelling cycle $\overline{a}_{c}$ that 
is the landing point of the external rays $R^{\theta_1}$ and
$R^{\theta _2}$. The following result gives the dynamical meaning of
the parameter values in the wake $W_\hh$ bounded by parameter external rays
$R_{\theta _{1}}$ and $R_{\theta _{ 2}}$ (See \cite{Sch} for further details).

\begin{thm}
For parameters $c$ in $W_\hh \setminus $\{root\}, the cycle $\overline{a}_c$ stays repelling 
and, moreover, the isotopy type of the ray portrait $\mfr (\overline{a}_{c})$ is fixed throughout $W_\hh$.
\end{thm}

\subsection{Polynomial-like mappings}
A polynomial-like map is a holomorphic proper branched covering of degree $d$, $f:U'\rightarrow U$, where $U$ and $U'$ are simply connected domains with $U'$ compactly contained in $U$. For example, every polynomial can be viewed as a polynomial-like mapping once restricted to an appropriate neighborhood of its filled Julia set. This notion was introduced in \cite{DH2} to explain the presence of homeomorphic copies of the Mandelbrot set
within the Mandelbrot set.

The filled Julia set $K(f)$ of a polynomial-like mapping $f:U'\to U$ is naturally defined as
\[K(f):=\{ z\in\mathbb{C} : f^{n}(z) \in U', \text{ for } n=0, 1, 2, \ldots\}.\]
The Julia set $J(f)$ is defined as the boundary of $K(f)$. These sets are
connected if and only if $K(f)$ contains all critical points of $f$. 

Two polynomial-like mappings $f$ and $g$ with Julia sets $J(f)$ and $J(g)$, respectively, are called 
\textit{topologically} \textit{conjugate} if there are choices of domains $U$, $U'$, $V$, and $V'$ as 
well as a homeomorphism $h:U \ra V$ such that $f:U'\rightarrow U$ and $g:V'\rightarrow V$ are 
polynomial-like, with the same Julia sets $J(f)$ and $J(g)$, and $h\circ f= g\circ h$ on $U'$.
They are called \textit{quasi-conformally} (\textit{conformally}, or \textit{affinely}) conjugate if $h$ can be 
chosen quasi-conformal (conformal, or affine, respectively). The notation $\Dil (h)$ is used for the 
quasi-conformal dilatation of a given quasi-conformal mapping $h$. 

Two polynomial-like mappings $f$ and $g$ are \textit{hybrid or internally equivalent} if there exists a 
quasi-conformal conjugacy $h$ (q.c.\ conjugacy for short) between $f$ and $g$ such that
$\overline{\partial}h=0$ on $K(f)$. The following remarkable rigidity type theorem due to
Douady and Hubbard \cite{DH2} states that the dynamics of a polynomial-like mapping is essentially the same as the one of a polynomial.

\begin{thm}[Straightening]
Every polynomial-like mapping $f$ is hybrid equivalent to (a suitable
restriction of) a polynomial $P$ of the same degree. Moreover, $P$
is unique up to affine conjugacy when $K(f)$ is connected.
\end{thm}

From now on we only consider polynomial-like mappings with only one branched 
point of degree $d$, assumed to be at zero by normalization, and refer to them as 
\textit{unicritical polynomial-like} mappings. By the above theorem, any unicritical polynomial-like 
mapping with connected
Julia set corresponds to a unique (up to affine conjugacy)
unicritical polynomial $z \mapsto z^{d}+c$, with $c$ in $\M_d$. 
Note that $z^d+c$ and $z^d+ c/\lambda$ are conjugate via $z \mapsto \lambda z$ for every $(d-1)$-th root 
of unity $\lambda$.

Given a polynomial-like mapping $f:U'\rightarrow U$, we can consider
the \textit{fundamental annulus} $A=U\setminus U'$. 
It is not canonical because any choice of $V'\Subset V$ such that
$f:V'\rightarrow V$ is a polynomial-like mapping with the same Julia set
gives a different annulus. However, we can associate a real number,
the \textit{modulus} of $f$, to any polynomial-like mapping $f$ as follows:
\[\mathrm{mod} (f)=\sup \mathrm{mod} (A),\] 
where the supremum is taken over all possible
fundamental annuli $A$ of $f$.

It is easy to see that the hybrid conjugacy obtained in the straightening theorem is 
not unique. However, given a polynomial-like mapping, one can build a hybrid conjugacy as in the 
straightening Theorem with a uniform bound on its dilatation in terms of the modulus of the polynomial-like mapping.
As this is essential in the rest of this work, we formulate it in the following proposition.
\begin{prop}
For every $\eta > 0$, there exists a constant $K>0$, such that if $f$ is a polynomial-like mapping with 
$mod(f)\geq \mu$, then one can choose a hybrid conjugacy as in the straightening theorem whose dilatation is 
bounded by $K$.
\end{prop}
\section{Modified principal nest}
\subsection{Yoccoz puzzle pieces}\label{puzzles}
Recall that for a parameter $c\in \M_d$ outside of the main hyperbolic component, $P_{c}$ has a unique 
dividing fixed point $\alpha_{c}$. 
The $q\geq 2$ external rays landing at $\alpha_c$ together with an arbitrary equipotential $E^r$ cut the domain inside $E^r$ into $q$ closed \textit{topological disks} (i.e.\ simply connected domains in $\mathbb{C}$) $Y^{0}_{j}, j=0, 1,\ldots, q-1$, called \textit{puzzle pieces of level zero}. That is, $Y^0_j$'s are the closures of the bounded components of 
$\mathbb{C}\setminus \{E^r\cup \overline{\mfr(\alpha_c)}\}$. The main property of this partition is 
that each $P_{c}(\partial Y_{j}^{0})$ does not intersect the interior of any piece $Y_{i}^{0}$.

\textit{Puzzle pieces $Y_i^n$ of level or depth $n$} are defined as the closures of the connected components 
of $P_c^{-n}(\daron (Y^{0}_{j}))$. They partition the region bounded by the equipotential $P_c^{-n}(E^r)$ into a 
finite number of closed disks. By definition, all puzzle pieces are bounded by piecewise analytic curves. 
The puzzle piece of level $n$ containing the critical point is referred to as the \textit{critical puzzle piece} 
of level $n$.  
The \textit{label} of a puzzle piece is the set of the angles of the external rays bounding that puzzle
piece. 
If the critical point does not land on $\alpha_c$, there is a unique critical puzzle piece 
$Y^n_0$ of every level $n$.

The family of all puzzle pieces of $P_c$ has the following \textit{Markov property}:

\begin{itemize}
\item[--] Any two puzzle pieces are either disjoint or nested. 
In the latter case, the puzzle piece of higher level is contained in the puzzle piece of lower level.
\item[--] Image of any puzzle piece of level $n \geqslant 1$ is a puzzle piece of level $n-1$. 
Moreover, $P_c:Y^{n}_{j} \rightarrow Y^{n-1}_{k}$ is either $d$-to-$1$ branched covering, or univalent. 
This depends on whether $Y^{n}_{j}$ contains the critical point or not.
\end{itemize}

On the first level, there are $d(q-1)+1$ puzzle pieces organized as follows. 
The critical piece $Y^{1}_{0}$; the $q-1$ (off critical) pieces attached to the fixed point $\alpha _{c}$ that are denoted by $Y^{1}_{i}$, for $i=1, 2, \dots, q-1$; and the symmetric ones attached to $P_c^{-1}(\alpha_c)\setminus \{\alpha_c\}$ that are denoted by $Z^{1}_{i}$, for $i=1, 2, \dots, (d-1)(q-1)$. Moreover, $P_c|Y^{1}_{0}$,  $d$-to-$1$ covers $Y^{1}_{1}$, $P_c|Y^{1}_{i}$ univalently covers $Y^{1}_{i+1}$, for every $i=1,\ldots, q-2$, and $P_c|Y^{1}_{q-1}$ univalently covers $Y^{1}_{0}\cup \bigcup_{i=1}^{(d-1)(q-1)} Z^{1}_{i}$. 
Thus, $P_c^{q}(Y_{0}^{1})$ truncated by $P_c^{-1}(E^r)$ is equal to the union of $Y_{0}^{1}$ and ${Z^{1}_{i}}$'s.

From now on we assume that $P_c^n(0)\neq \alpha_c$, for all $n$. 
Therefore, the critical puzzle piece of every level is uniquely determined. 
As it will be apparent in a moment, this condition is always the case for the parameters we are interested in.
\subsection{Favorite nest and renormalization}
\label{favnest}
Given a puzzle piece $V$ containing $0$, let $R_{V}:\Dom R_{V}\subseteq V\rightarrow V$ denote 
\textit{the first return map} to $V$. It is defined at every point $z$ in $V$ for which 
there exists a positive integer $t$ with $P_c^t(z) \in \daron V$. 
For every such $z$, $R_{V}(z)$ is 
defined as $P_c^t (z)$, where $t$ is the smallest positive integer with $P_c^t(z) \in \daron V$.  
The Markov property of the puzzle pieces implies that any component of $\Dom R_{V}$ is contained in $V$, 
and moreover, the restriction of this return map ($P_{c}^{t}$, for some $t$) to such a component 
is either a $d$-to-$1$, or $1$-to-$1$ proper map onto $V$. 
The component of $\Dom R_V$ containing the critical point  is called 
the \textit{central component} of $R_{V}$. If image of the critical point under the first 
return map belongs to the central component, the return is called \textit{central}.

The \textit{first landing map} $L_{V}$ to a puzzle piece $V\ni0$ is defined at all points $z\in \mathbb{C}$ for which
 there exists an integer $t\geq 0$ with $P_c^t(z)\in \daron V$. 
It is the identity on $V$, and it univalently maps each component of $\Dom L_{V}$ onto $V$.

Consider a puzzle piece $Q\ni0$. If the orbit of the critical point returns back to $Q$ under iterates of $P_c$, 
the central component $P\subset Q$ of $R_{Q}$ is the pullback of $Q$ by $P_{c}^{m}$ along the orbit 
$0,P_c(0), \dots, P_c^m(0)$, where $m$ is the first moment when the critical orbit enters $\daron Q$. 
Hence, $P_{c}^{m}:P\rightarrow Q$ is a proper map of degree $d$. This puzzle piece $P$ is called the 
\textit{first child} of $Q$.

The \textit{favorite child} $Q'$ of $Q$ is constructed as follows: Let $p>0$ be the first moment when 
$R^{p}_{Q}(0) \in \daron (Q\setminus P)$ (if it exists). 
Now, let $q>0$ be the first moment (if it exists) when $R^{p+q}_{Q}(0)\in \daron P$. 
In other words, $p+q$ is the moment of the first return back to $P$ after 
the first escape of the critical point from $P$ under iterates of $R_{Q}$. 
Now, $Q'$ is defined as the pullback of $Q$ under $R^{p+q}_{Q}$ containing the critical point. 
The Markov property of the puzzle pieces implies that the map $R_{Q}^{p+q}=P^{k}_{c}:Q'\ra Q$ 
(for an appropriate $k>0$) is a proper map of degree $d$. 
The main property of the favorite 
child is that the image of the critical point under $P^{k}_{c}:Q'\rightarrow Q$ belongs to the first child $P$.

Let $P_c$ be a unicritical polynomial with $q>1$ external rays landing at its $\alpha$-fixed point, and form 
the corresponding Yoccoz puzzle pieces introduced in Section \ref{puzzles}. 
The map $P_c$ is called \textit{satellite renormalizable}, (also called \textit{immediately 
renormalizable} by Douady and Hubbard) if 
\[P_c^{lq}(0)\in Y^{1}_0,\quad\textrm{for} \;l=0, 1,  2,  \ldots .\]
The map $P_c^q:Y^1_0 \rightarrow P_c^q(Y^1_0)$ is a proper branched covering of \text{degree $d$}. 
However, its domain is not compactly contained in its range. One can slightly enlarge \label{thickening} $Y^1_0$ 
so that it is compactly contained in its range (see \cite{Mi3} for a detailed argument). 
Thus, $P_c^q$ can be turned into a unicritical polynomial-like mapping. Note that the above condition 
on the orbit of the critical point implies that the corresponding little Julia set is connected.

If $P_c$ is not satellite renormalizable, then there is the first positive integer $k$ such that $P_c^{kq}(0)$ 
belongs to some $Z^{1}_{i}$. Define $Q^{1}$ as the pullback of this $Z^{1}_{i}$ under $P_c^{kq}$ containing 
the critical point. 
By the above process we form the first child $P^{1}$ and the favorite child $Q^{2}$ of $Q^{1}$. 
Repeating the above process we obtain a (finite or infinite) nest of puzzle pieces
\begin{align}
Q^{1} \supset P^{1} \supset Q^{2} \supset P^{2} \supset\dots \supset
Q^{n} \supset P^{n} \supset\cdots %\tag{nest}
\label{nest}
\end{align}
where $P^{i}$ is the first child of $Q^{i}$, and $Q^{i+1}$ is the favorite child of $Q^{i}$. 

The above nest is finite if and only if one of the following happens:
\begin{itemize}
\item[--] The map $P_c$ is combinatorially non-recurrent,  that is, the critical point does not return to 
some critical puzzle piece.
\item[--] The orbit of the critical point does not escape some $P^{n}$ under iterates of $R_{Q^{n}}$, 
or equivalently, the first return maps to all the critical puzzle pieces of level bigger than some $n$ are central.
\end{itemize}

Combinatorial rigidity of the combinatorially non-recurrent parameters has been taken care of in \cite{Mi3}. 
In the latter case, $R_{Q^n}=P^k_c: P^n \rightarrow Q^n$ (for an appropriate $k$) is a unicritical 
polynomial-like mapping of degree $d$ with $P^n$ compactly contained in $Q^n$. 
The map $P$ is called \textit{primitively renormalizable} in this case. 
Note that the corresponding little Julia set is connected because all the returns
of the critical point to $Q^n$ are central by definition.

A unicritical polynomial is called \textit{renormalizable} if it is satellite or primitively
renormalizable.

\subsection{Complex bounds and pseudo-conjugacies}\label{complex-bound}
The general strategy, starting with Yoccoz's work on quadratics \cite{H}, to prove the local connectivity of 
some  Julia sets and rigidity of complex unicritical polynomials has been showing that every nest of puzzle 
pieces shrink to a point. 
To deal with non-renormalizable and combinatorially recurrent polynomials, the following 
\textit{a priori} bounds property has been proved in \cite{AKLS}.

\begin{thm}\label{priori bounds}
There exists a constant $\delta >0$ such that for every $\eps >0$ there exists $n_0>0$, with the following property. In the nest of puzzle pieces \eqref{nest}, if\, $\mod(Q^{1}\setminus P^{1})> \eps$, then  for all $n\geq n_0$ we have $\mod(Q^{n} \setminus P^{n})>\delta$.
\end{thm}

If $P_c$ is combinatorially recurrent, the orbit of the critical point does not land at $\alpha$-fixed point. 
Therefore, puzzle pieces of all levels are well defined. 
Now, let $P_c$ be a non-renormalizable unicritical polynomial. 
The \textit{combinatorics of $P_c$} up to level $n$ is defined as an equivalence relation on the set of labels 
of all puzzle pieces of level less than or equal to $n$. 
Two such angles $\theta_1$ and $\theta_2$ are considered equivalent if $R^{\theta_1}$ and 
$R^{\theta_2}$ land at the same point. 
One can see that the combinatorics of a map up to level $n+t$ determines the puzzle piece $Y^{n}_{j}$ of level 
$n$ containing the critical value $P_c^{t}(0)$, for all positive integers $n$ and $t$. 
Two non-renormalizable maps are called \textit{combinatorially equivalent} if they have the same combinatorics 
up to an arbitrary level $n$. 
The combinatorics of a renormalizable map will be defined in Section \ref{combinatorics}.

Two unicritical polynomials $P_c$ and $P_{\tilde{c}}$ with the same combinatorics up to some level $n$ 
are called \textit{pseudo-conjugate up to level $n$} if there exists an orientation preserving homeomorphism 
$H:(\mathbb{C},0) \rightarrow (\mathbb{C},0)$, such that $H(Y_{j}^{0})=\widetilde{Y}^{0}_{j}$, for all $j$, 
and $H\circ P_c=P_{\tilde{c}} \circ H$ outside of the critical puzzle piece $Y_{0}^{n}$. 
Such a pseudo-conjugacy $H$ is said to \textit{match the B\"{o}ttcher marking}, if near infinity it becomes 
the identity in the B\"{o}ttcher coordinates for $P_c$ and $P_{\tilde{c}}$. 
Thus, by the equivariance property of a pseudo-conjugacy, it is the identity map in the B\"{o}ttcher coordinates 
outside of $\cup _{j} Y_j^{n}$.

Let $q_{m}$ and $p_{m}$ denote the levels of the puzzle pieces $Q^{m}$ and $P^{m}$, respectively, 
in the nest \eqref{nest}, that is, \text{$Q^{m}=Y^{q_{m}}_0$}, and \text{$P^{m} =Y^{p_{m}}_0$}. 
The following theorem is the main technical result of \cite{AKLS} which is frequently used in the proof of 
our main theorem.

\begin{thm} \label{pseudo-conj}
Assume that a nest of puzzle pieces 
\begin{equation} \label{finite-nest}
Q^{1}\supset P^{1}\supset Q^{2}\supset P^{2}\supset\dots\supset Q^{m}\supset P^{m}
\end{equation} 
is obtained for $P_c$, and $P_{\tilde{c}}$ is combinatorially equivalent to $P_c$ up to level $q_m$, 
where $Q^{m}=Y^{q_{m}}_0$. Then there exists a $K$-q.c.\ pseudo-conjugacy $H$ up to level $q_{m}$ between $P_c$ 
and $P_{\tilde{c}}$ which matches the B\"{o}ttcher marking.
\end{thm}
To control the dilatation of the pseudo-conjugacy obtained in this theorem, we show the following statement.
{\samepage
\begin{prop}\label{unifrom-pseudo-conj}
Assume that the nest of puzzle pieces in the above theorem is defined using the equipotential of level $\eta$, 
then the dilatation of the q.c.\ pseudo-conjugacy obtained in that theorem depends only on the hyperbolic distance
between $c$ and $\tilde{c}$ in the primary wake (containing $c$ and $\tilde{c}$) truncated by the parameter
equipotential of level $\eta$, and the modulus of the annulus $Q^1\setminus P^1$.
\end{prop}
}
\begin{proof}
To prove the proposition, we need a brief sketch of the proof of the above theorem. For more details one may 
refer to \cite{AKLS}. 

The combinatorial equivalence of $P_c$ and $P_{\tilde{c}}$ up to level zero implies that $c$ 
and $\tilde{c}$ belong to the same truncated wake $W(\eta)$ attached to the main hyperbolic component of $\M_d$. 
Inside $W(\eta)$, the $q$ external rays $\mathfrak{R}(\alpha)$ and $E^h$, for any $h >\eta$, 
move holomorphically in $\mathbb{C}\setminus 0$. 
That is, there exists a holomorphic motion 
$\Phi:W(\eta)\times \{\mathfrak{R}(\alpha)\cup E^h\}\ra W(\eta)\times \cc$, given by 
$B_{\tilde{c}}^{-1} \circ B_c$ in the second coordinate, such that
\[\Phi(c,\mathfrak{R}(\alpha)\cup E^h)=(\tilde{c},\mathfrak{R}(\widetilde{\alpha})\cup \widetilde{E}^h).\]

Outside of the equipotential $E^h$, this holomorphic motion extends to the motion holomorphic in both variables 
$(c,z)$ which is obtained from the B\"{o}ttcher coordinates near $\infty$. 
By \cite{Sl} the map $\Phi(\tilde{c},\cdot)\circ \Phi(c,\cdot)^{-1}$ extends to a $K_{0}$-q.c.\ homeomorphism 
$G_0:(\mathbb{C},0) \ra (\mathbb{C},0)$, where $K_{0}$ depends only on the hyperbolic distance between $c$ 
and $\tilde{c}$ in $W(\eta)$. 
The map $G_0$ conjugates $P_c$ to $P_{\tilde{c}}$ outside of the puzzle pieces of level zero.

By adjusting $G_0$ inside the equipotential $E^h$ such that it sends $c$ to $\tilde{c}$, 
we obtain a q.c.\ homeomorphism (not necessarily with the same dilatation) $G_0'$. 
By lifting $G_0'$ via $P_c$ and $P_{\tilde{c}}$ we obtain a new q.c.\ homeomorphism $G_1$. 
That is, $G_1$ is the unique map satisfying $ P_{\tilde{c}} \circ G_1= G_0'  \circ  P_c$.
 
Now, we repeat the following two processes, for $i=2, 3, \dots, n=q_m$, 
\begin{itemize}
\item[--] Adjust the q.c.\ homeomorphism $G_{i-1}$ inside the union of puzzle pieces of level $i$ so that 
it sends $c$ to $\tilde{c}$;
\item[--] Lift the adjusted map via $P_c$ and $P_{\tilde{c}}$ to obtain a q.c.\ homeomorphism $G_{i}$ 
(not necessarily with the same dilatation) conjugating $P_c$ to $P_{\tilde{c}}$ outside the union of puzzle pieces of level $i$. 
\end{itemize}

At the end, we obtain a q.c.\ homeomorphism $G_n$ which conjugates $P_c$ to $P_{\tilde{c}}$ outside of the equipotential 
$E^{h/d^n}$.

The nest of puzzle pieces 
\[\widetilde{Q}^{1}\supset \widetilde{P}^{1}\supset \widetilde{Q}^{2}\supset \widetilde{P}^{2}\supset\dots\supset \widetilde{Q}^{m}\supseteq \widetilde{P}^{m}\] 
for $P_{\tilde{c}}$ is defined as the image of the nest \eqref{finite-nest} under $G_n$. 
The combinatorial equivalence of $P_c$ and $P_{\tilde{c}}$ implies that this new nest has the same properties as 
the one of $P_c$. 
In other words, $\widetilde{Q}^{i+1}$ is the favorite child of $\widetilde{Q}^i$, and $\widetilde{P}^i$ is 
the first child of $\widetilde{Q}^i$. 
Hence, Theorem \ref{priori bounds} applies to this nest as well. 
By properties of these nests, one constructs a $K$-q.c.\ homeomorphism $H_n$ from the critical puzzle piece $Q^{n}$ 
to $\widetilde{Q}^{n}$, where $K$ depends only on the \textit{a priori} bounds $\delta$ and the hyperbolic 
distance between $c$ and $\tilde{c}$ in $W(\eta)$. 
The pseudo-conjugacy $H_{n}$ is obtained from univalent lifts of $H_n$ onto other puzzle pieces.
\end{proof}

If $P_c$ is renormalizable, the process of constructing the modified principal nest stops at some level, and all 
the returns to the critical puzzle pieces of higher level become central. 
One can see that in this situation the critical puzzle pieces do not shrink to 0.

\subsection{Combinatorics of a map}\label{combinatorics}
If a map $P_{c_0}$ is renormalizable, there is a unique maximal homeomorphic copy $\M^{1}_{d} \ni c_{0}$ of the
connectedness locus within the connectedness locus satisfying the following properties (see\cite{DH2}): For
$c\in\M^{1}_{d}\setminus \{\textrm{root}\}$, $P_{c}:z \mapsto z^{d}+c$ is renormalizable, and in addition, 
there is a holomorphic motion of the dividing fixed point $\alpha _{c}$ and the rays landing at it on a 
neighborhood of $\M^{1}_{d} \setminus \{\textrm{root}\}$, such that the renormalization of $P_{c}$ is 
associated to this fixed point and the external rays landing at it. 
Furthermore, all parameters in this copy have Yoccoz puzzle pieces of all levels with the same labels. 
This copy is maximal in the sense that it is not contained in any other copy except the actual connectedness locus. 
The homeomorphism from the copy to $\M_d$ is not unique because of the rotational symmetry of $\M_d$. 
However, we make it unique by sending the sole root point of the copy to the landing point of the parameter 
external ray of angle zero. 
We denote this (first) renormalization of $P_c$ by $\rr P_c$.

Assume that $\rr P_c$ is given as $P_c^j:U \rightarrow U'$, for some positive integer $j$ and topological 
disks $U$ and $U'$. 
By the straightening theorem, $\rr P_c$ is conjugate to a unicritical polynomial $P_{c'}$.  
The polynomial $P_{c'}$ is determined up to conformal equivalence in this theorem. 
However, there are only $d-1$ polynomials in each conformal class (these are $c'\cdot \lambda$, for $\lambda$
with $\lambda^{d-1}=1$).
We make this parameter unique by choosing the image of $c$ under the unique  homeomorphism from the copy to 
the connectedness locus determined above.   

If $P_{c'}$ is also renormalizable, $P_c$ is called \textit{twice renormalizable}. 
Let the positive integer $k$, and topological disks $V$ and $V'$ be such that \text{$P_{c'}^k:V\rightarrow V'$} 
gives the first renormalization of $P_{c'}$ determined as above. 
Define the topological disks $\widetilde{V}$ and $\widetilde{V}'$ as $\chi$-preimages of $V$ and $V'$, respectively, 
where $\chi$ is a straightening of $\rr P_c$. 
One can see that $\chi$ conjugates $P_c^{jk}:\widetilde{V}\ra \widetilde{V}'$ to $P_{c'}^k:V \rightarrow V'$. 
Therefore, $P_c^{jk}:\widetilde{V}\rightarrow \widetilde{V}'$ is also polynomial-like. 
We denote this map by $\rr^{2} P_c$. 

The above process may be continued to associate a finite or an infinite sequence $P_c,\rr P_c,\rr^{2}P_c,\ldots$, 
of polynomial-like mappings to $P_c$, and accordingly, call $P_c$ \textit{at most finitely} or 
\textit{infinitely renormalizable}. 
Let $P_{c_1}, P_{c_2}, P_{c_3},\dots$ denote the polynomials obtained from straightening the 
polynomial-like mappings $\langle\rr^n P_c\rangle_{n=0}$. 
We associate the finite or infinite sequence 
\[\tau(P_c):=\langle\M^1_d, \M^2_d, \ldots\rangle,\] 
of the maximal copies of the locus to $P_c$, where $\M^n_d$ corresponds to the renormalization 
$\rr P_{c_{n-1}}$. 
Earlier in Section \ref{complex-bound} we defined the combinatorics of a non-renormalizable unicritical 
polynomial as the equivalence relation on the labels of the Yoccoz puzzle pieces. 
It turns out that all the parameters in a given copy of the connectedness locus within the parameter space have 
the same combinatorics in this sense. 
To further refine our definition of the combinatorics, one may consider the same equivalence relation, i.e.\ 
landing at the same point, on a larger set of angles of external rays. 
Here, for an infinitely renormalizable $P_c$, the sequence $\tau(P_c)$ is called the \textit{combinatorics} of $P_c$. 
This definition, which is chosen for our convenience, is equivalent to the above definition of the combinatorics
when we consider the equivalence relation on the set of angles of all periodic external rays.     

Hence, two infinitely renormalizable maps are called \textit{combinatorially equivalent} if they have the 
same combinatorics, i.e.,  correspond to the same sequence of maximal connectedness locus copies. 

We say that an infinitely renormalizable $P_c$ satisfies the \textit{secondary limbs condition}, if all the parameters 
$c_1, c_2, \dots$, obtained from straightening the \text{polynomial-like} mappings $\{\rr^n P_c\}_{n=0}^\infty$ belong to 
a finite number of truncated secondary limbs.
Let $\mathcal{SL}$ stand for the class of infinitely renormalizable unicritical polynomial-like mappings satisfying the 
secondary limbs condition.

An infinitely renormalizable map $P_c$ is said to satisfy \textit{a priori bounds}, if there exists an $\eps >0$ with 
$\mod (\rr^{m}P_c)\geq \eps$, for all $m\geq 1$.
\section{Proof of the rigidity theorem}
\subsection{Reductions}
{\samepage
\begin{thm}\label{THM}
Let $f$ and $\tilde{f}$ be two infinitely renormalizable unicritical polynomial-like mappings satisfying the 
$\mathcal{SL}$ and \textit{a priori} bounds conditions. 
If $f$ and $\tilde{f}$ are combinatorially equivalent, then they are hybrid equivalent.
\end{thm}
}
\begin{rem}
In particular, if the two maps $f$ and $\tilde{f}$ in the above theorem are polynomials, then hybrid 
equivalence becomes conformal equivalence. 
That is because the identity map in the B\"{o}ttcher coordinates, which conformally conjugates the two maps on 
the complements of their Julia sets, can be glued to the hybrid conjugacy on the Julia set. 
See Proposition $6$ in \cite{DH2} for a precise proof of this statement.
\end{rem}

The proof of this theorem breaks into the following steps:
\begin{gather*}
\text{combinatorial equivalence}\\
\Downarrow \\
\text{topological equivalence}\\
\Downarrow \\
\text{q.c.\ equivalence} \\
\Downarrow  \\
\text{hybrid equivalence}
\end{gather*}

It has been shown in \cite{Ji} that any \textit{unbranched} infinitely renormalizable map with \textit{a priori} bounds has a locally connected Julia set. 
The renormalizations $f^{n_k}:U_k\ra V_k$, for $k=1, 2, 3,\dots$, are said to satisfy the unbranched property if the domains 
$U_k$ and $V_k$ which provide the \textit{a priori} bounds also satisfy 
\[\pc(f)\cap U_k=\pc(f^{n_k}:U_k\ra V_k), \text { for } k=1, 2, 3, \dots .\] 
Here, the unbranched property follows from our combinatorial and \textit{a priori} bounds conditions (see \cite{Ly97}, Lemma $9.3$). Then, the first step, topological equivalence of combinatorially equivalent maps, follows from the
local connectivity of the Julia sets by Carath\'{e}odory's Theorem. 
That is, the identity map in the B\"{o}ttcher coordinates extends onto the Julia set. 
Indeed, by \cite{Do93}, there is a topological model for the Julia sets of these maps based on their combinatorics.   

The last step follows from McMullen's rigidity Theorem (\cite{Mc1}, Theorem $10.2$). 
He has shown that an infinitely renormalizable quadratic 
polynomial-like mapping with \textit{a priori} bounds does not support any nontrivial invariant line fields on 
its Julia set. 
The same proof works for unicritical polynomial-like mappings of any degree. 
It follows that any q.c.\ conjugacy $h$ between $f$ and $\tilde{f}$ satisfies $\overline{\partial}h=0 $ 
almost everywhere on the Julia set. Therefore, $h$ is a hybrid conjugacy between $f$ and $\tilde{f}$. 
However, if all infinitely renormalizable unicritical maps in a given combinatorial class satisfy the 
\textit{a priori} bounds condition, there is an easier way to show that q.c.\ conjugate maps are hybrid 
conjugate in that class. 
Since we are going to apply our theorem to combinatorial classes for which \textit{a priori} bounds has 
been established, we will prove this in Proposition~\ref{openclosed}.

So assume that $f$ and $\tilde{f}$ are topologically conjugate. We want to show the following:

\begin{thm}
Let $f$ and $\tilde{f}$ be infinitely renormalizable unicritical polynomial-like mappings satisfying the 
\textit{a priori} bounds and $\mathcal{SL}$ conditions. If $f$ and $\tilde{f}$ are topologically conjugate 
then they are quasi-conformally conjugate.
\end{thm}
Given sets $A\subseteq B\subseteq C$ and $\widetilde{A}\subseteq \widetilde{B}\subseteq \widetilde{C}$, the notation $h:(C, B, A)\ra (\widetilde{C}, \widetilde{B}, \widetilde{A})$ means that $h$ is a map from $C$ to $\widetilde{C}$ with $h(B)=\widetilde{B}$ and $h(A)=\widetilde{A}$.

\subsection{Thurston equivalence}
Suppose that two unicritical polynomial mappings $f:U_2\ra U_1$ and $\tilde{f}:\widetilde{U}_2 \ra \widetilde{U}_1$ are topologically conjugate. A q.c.\ homeomorphism 
\[h:(U_1, U_2,\pc(f)) \ra (\widetilde{U}_1,\widetilde{U}_2,\pc(\tilde{f}))\] 
is called a \textit{Thurston conjugacy} if it is homotopic, relative $\partial U_1 \cup \partial U_2 \cup  \pc(f)$, to a topological conjugacy 
\[\psi:(U_1, U_2,\pc(f)) \ra (\widetilde{U}_1, \widetilde{U}_2, \pc(\tilde{f}))\] 
between $f$ and $\tilde{f}$. 
Note that a Thurston conjugacy does not conjugate the two maps. 
It is a conjugacy on the post-critical set, and homotopic to a conjugacy on the complement of the post-critical set. 

The following result due to Thurston and Sullivan~\cite{S1} originates the ``pullback method'' in holomorphic dynamics.
\begin{lem}
Thurston conjugate unicritical polynomial-like mappings are q.c.\ conjugate.
\end{lem}
The proof given in \cite{Ly97}, Lemma~10.1, in the quadratic case works for the unicritical maps without any change.

By the \textit{a priori} bounds assumption in the theorem, there are topological disks $V_{n,0}\Subset U_{n,0}$
containing zero such that $\rr^n f:=f^{t_n}: V_{n,0}\ra U_{n,0}$, for $n=1, 2, \dots$, are unicritical 
degree $d$ polynomial-like with $\mod (U_{n,0}\setminus V_{n,0})\geq \eps$, for some $\eps>0$. 
By going several levels down, i.e., considering $f^{t_n}: f^{-{k_n}t_n}(V_n)\ra f^{-{k_n}t_n}(U_n)$, for some positive integers $k_n$, we may assume that $\mod (U_n\setminus V_n)$ and 
$\mod (\widetilde{U}_n\setminus \widetilde{V}_n)$ are uniformly bounded from above and are \textit{comparable}. 
The latter means that, there exists a constant 
$M>0$ such that for every $n\geq 1$,   
\begin{equation}\label{const-M}
\frac{1}{M}\leq\frac{\mod (U_n\setminus V_n)}{\mod (\widetilde{U}_n\setminus\widetilde{V}_n)}\leq M.
\end{equation} 
Also by slightly shrinking the domains $U_{n,0}$, if necessary, we may assume that they have smooth boundaries.
Hence, we can assume the following in the remainder of this note.  
\begin{itemize}\label{nice-properties}
\item[--]There exist positive constants $\eps$ and $\eta$ such that for every $n\geq 1$, 
\begin{equation}\label{const-eta}
\eps \leq \mod (U_{n,0}\setminus V_{n,0}) \leq \eta, \text{ and }\eps\leq \mod (\widetilde{U}_{n,0}\setminus \widetilde{V}_{n,0})\leq \eta;
\end{equation}
\item[--]For every $n\geq 1$, $U_{n,0}$ and $\widetilde{U}_{n,0}$ have smooth boundaries;
\item[--] There exists a constant $M>0$, for which \eqref{const-M} holds for every $n\geq1$.
\end{itemize}

We use the following notations throughout the rest of this note.
\begin{align*}
&f:V_0\ra U_0, J_{0,0}=J(f), \\
&\rr^{n}f=f^{t_{n}}:V_{n,0}\ra U_{n,0}, J_{n,0}=J(\rr^{n}f),\text{ for } n\geq 1 \text{ and } t_n\geq 1.
\end{align*}
The domain $V_{n,i}$, for $i=1,2,\dots,t_n-1$, is defined as the preimage of $V_{n,0}$ under $f^{i}$ containing 
the \textit{little Julia set} $J_{n,i}:=f^{t_n-i}(J_{n,0})$. 
Similarly, $U_{n,i}$, for $i=1, 2, \dots, t_n-1$, is defined as the component of $f^{-i}(U_{n,0})$ 
containing $V_{n,i}$ so that $f^{t_n}:V_{n,i}\ra U_{n,i}$ is polynomial-like of degree $d$. 
The domain $W_{n,i}$ is defined as the preimage of $V_{n,i}$ under the map $f^{t_n}:V_{n,i}\ra U_{n,i}$. 

Note that $\rr^{n}f : V_{n,i}\ra U_{n,i}$ is a polynomial-like mapping of degree $d$ with the Julia set 
$J_{n,i}$, and is conjugate to $\rr^{n}f:V_{n,0}\ra U_{n,0}$ by the conformal isomorphism 
\text{$f^{i}:U_{n,i}\ra U_{n, 0}$}. 

It has been proved in \cite{Ly97} (Lemma~9.2) that for the parameters satisfying our 
assumptions, the little Julia sets on the primitive levels are \textit{well apart}. This means that for 
every $n\geq1$ with $\rr^{n-1} f$ primitively renormalizable, one can choose pairwise disjoint 
domains $U_{n,i}$, for $i=1, 2, \dots, t_n-1$, for the renormalizations with moduli of the annuli $U_{n,i} \setminus V_{n,i}$ uniformly away from zero independent of $n$ and $i$. 
So, we will assume that on the primitive levels, the domains $U_{n, i}$ are disjoint for different values of $i$.
\begin{remNot}
From now on, any notation introduced for $f$ will be automatically introduced for $\tilde{f}$, and marked with a tilde. 
\end{remNot}
To build a Thurston conjugacy, first we introduce multiply connected domains $\Omega_{n(k), i}$ 
(and $\widetilde {{\Omega}}_{n(k),i}$) in $\mathbb{C}$, for an appropriate subsequence $n(1)< n(2)< n(3)< \cdots$ 
of the renormalization levels and $0\leq i\leq t_{n(k)}-1$, as well as a sequence of q.c.\ homeomorphisms 
\[h_{n(k),i}:\Omega_{n(k),i}\ra \widetilde {{\Omega}}_{n(k),i},\]
for $k=0,1,2, \dots$ and $i=0,1,2,\dots, t_{n(k)}-1$, with uniformly bounded dilatations. 
These domains will satisfy the following properties (see Figure~\ref{multiple}):
\begin{itemize}
\item[--] Every $\Omega_{n(k),j}$ is a topological disk minus $\frac{t_{n(k+1)}}{t_{n(k)}}$ number of 
topological disks denoted by $D_{n(k+1),j+it_{n(k)}}$, for $i=0,1,\dots, \frac{t_{n(k+1)}}{t_{n(k)}}-1 $;
\item[--] Every $\Omega_{n(k),i}$ is \textit{well inside} $D_{n(k),i}$, that is, the moduli of the annuli obtained from $D_{n(k),i}\setminus \Omega_{n(k),i}$ are uniformly bounded away from zero independent of $n(k)$ and $i$;
\item[--] Every \textit{little post-critical set} $J_{n(k),i}\cap \pc(f)$ is well inside $D_{n(k),i}$;
\item[--] Every $D_{n(k),i}$ is the preimage of $D_{n(k),0}$ under $f^{i}$ containing $J_{n(k),i} \cap \pc(f)$.
Every $\Omega_{n(k),i}$ is the component of $f^{-i}(\Omega_{n(k),0})$ inside $D_{n(k),i}$.
\end{itemize}

Finally, we construct the Thurston conjugacy by appropriately gluing together the maps 
$h_{n(k),i}:\Omega_{n(k),i}\ra \widetilde {{\Omega}}_{n(k),i}$ on the complement of these 
multiply connected domains (which is a union of annuli). 
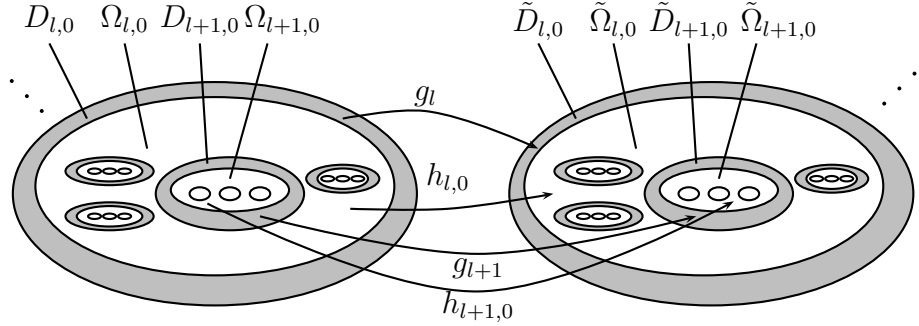
\begin{figure}[ht]
\begin{center}
  \begin{pspicture}(12,6)
  \psset{xunit=1cm}
  \psset{yunit=1cm}
 %  \psframe(0,0)(8.5,7.5)
       \psellipse[fillstyle=solid,fillcolor=lightgray](2.7,3)(2.7,1.5)              
       \psellipse[fillstyle=solid,fillcolor=white](2.7,3.1)(2.4,1.2) 

  \psellipse[fillstyle=solid,fillcolor=lightgray](2.9,3)(1,.5)
  \psellipse[fillstyle=solid,fillcolor=white](2.9,3.05)(.8,.3)
 
  \psellipse[fillstyle=solid,fillcolor=lightgray](1.3,3.3)(.6,.2) 
  \psellipse[fillstyle=solid,fillcolor=white](1.3,3.3)(.45,.15)
 
  \psellipse[fillstyle=solid,fillcolor=lightgray](1.3,2.7)(.6,.2)
  \psellipse[fillstyle=solid,fillcolor=white](1.3,2.7)(.45,.15)
 
  \psellipse[fillstyle=solid,fillcolor=lightgray](4.4,3.2)(.5,.2)
  \psellipse[fillstyle=solid,fillcolor=white](4.4,3.2)(.35,.15)

  \psellipse(2.9,3)(.15,.1)
  \psellipse(2.5,3)(.15,.1)
  \psellipse(3.3,3)(.15,.1)

  \psellipse(1.3,3.3)(.1,.05)
  \psellipse(1.1,3.3)(.1,.05)
  \psellipse(1.5,3.3)(.1,.05)

  \psellipse(1.3,2.7)(.1,.05)
  \psellipse(1.1,2.7)(.1,.05)
  \psellipse(1.5,2.7)(.1,.05)

  \psellipse(4.4,3.2)(.1,.05)
  \psellipse(4.6,3.2)(.1,.05)
  \psellipse(4.2,3.2)(.1,.05)
%the other one
  \psellipse[fillstyle=solid,fillcolor=lightgray](9.3,3)(2.7,1.5)              
  \psellipse[fillstyle=solid,fillcolor=white](9.2,3.1)(2.4,1.2) 

  \psellipse[fillstyle=solid,fillcolor=lightgray](9.4,3)(1,.5)
  \psellipse[fillstyle=solid,fillcolor=white](9.4,3.05)(.8,.3)
 
  \psellipse[fillstyle=solid,fillcolor=lightgray](7.8,3.3)(.6,.2) 
  \psellipse[fillstyle=solid,fillcolor=white](7.8,3.3)(.45,.15)
 
  \psellipse[fillstyle=solid,fillcolor=lightgray](7.8,2.7)(.6,.2)
  \psellipse[fillstyle=solid,fillcolor=white](7.8,2.7)(.45,.15)
 
  \psellipse[fillstyle=solid,fillcolor=lightgray](10.9,3.2)(.5,.2)
  \psellipse[fillstyle=solid,fillcolor=white](10.9,3.2)(.35,.15)

  \psellipse(9.4,3)(.15,.1)
  \psellipse(9,3)(.15,.1)
  \psellipse(9.8,3)(.15,.1)

  \psellipse(7.8,3.3)(.1,.05)
  \psellipse(7.6,3.3)(.1,.05)
  \psellipse(8,3.3)(.1,.05)

  \psellipse(7.8,2.7)(.1,.05)
  \psellipse(7.6,2.7)(.1,.05)
  \psellipse(8,2.7)(.1,.05)

  \psellipse(10.9,3.2)(.1,.05)
  \psellipse(11.1,3.2)(.1,.05)
  \psellipse(10.7,3.2)(.1,.05)

  \psline(.5,5)(1,4)        \rput(.5,5.3){$D_{l,0}$}
  \psline(1.5,5)(1.8,3.6)   \rput(1.5,5.3){$\Omega_{l,0}$}
  \psline(2.5,5)(2.6,3.4)   \rput(2.5,5.3){$D_{l+1,0}$}
  \psline(3.4,5)(2.9,3.2)   \rput(3.6,5.3){$\Omega_{l+1,0}$}
%the tilde side
  \psline(7,5)(7.5,4)        \rput(7,5.3){$\tilde{D}_{l,0}$}
  \psline(8,5)(8.3,3.6)   \rput(8,5.3){$\tilde{\Omega}_{l,0}$}
  \psline(9,5)(9.1,3.4)   \rput(9,5.3){$\tilde{D}_{l+1,0}$}
  \psline(9.9,5)(9.4,3.2)   \rput(10.2,5.3){$\tilde{\Omega}_{l+1,0}$}

  \pscurve{->}(4.4,4)(5.5,4.1)(7,3.6)  \rput(5.5,4.3){$g_l$}
  \pscurve{->}(4.5,2.8)(5.8,2.8)(7.2,3)  \rput(5.8,3.2){$h_{l,0}$}
  \pscurve{->}(3.3,2.7)(6.2,2.2)(9.1,2.7)  \rput(6.2,2){$g_{l+1}$}
  \pscurve{->}(2.6,2.87)(6.2,1.7)(9.6,2.9)  \rput(6.2,1.5){$h_{l+1,0}$}
  
  \psdots[dotsize=1.5pt](0,4.5)(.2,4.3)(.4,4.1)
  \psdots[dotsize=1.5pt](11.6,4.2)(11.8,4.4)(12,4.6)
\end{pspicture}
\caption{The multiply connected domains and the buffers}
\label{multiple}
\end{center}
\end{figure}

%\begin{figure}[ht]
%\begin{center} \scalebox{1}{\includegraphics{Rigidity/Figs/annuli}}
%\caption{The multiply connected domains and the buffers}
%\label{multiple}
%\end{center}
%\end{figure}
%%%%%%%%%%%%%%%%%%%%%%%%%%%%%%%%%%%%%%%%%%%%%%%%%%%%%%%%%%%%%%%%%%%%%%%%%
\subsection{The domains $\Omega_{n,j}$ and the maps $h_{n,j}$}
By straightening the polynomial-like mappings 
\[\rr^{n-1} f:V_{n-1,0} \ra U_{n-1,0},\] 
for $n=2, 3, 4, \dots $, we get $K_1(\eps)$-q.c.\ homeomorphisms $S_{n-1}$ as well as unicritical polynomials 
$\Bf_{c_{n-1}}$, such that 
\begin{gather}
S_{n-1}: (U_{n-1, 0}, V_{n-1, 0}, 0) \ra (\Upsilon^0_{n-1},
\Upsilon^1_{n-1}, 0), and \\
S_{n-1} \circ \rr^{n-1} f= \Bf_{c_{n-1}} \circ S_{n-1}. \notag
\label{stra}
\end{gather}
Note that $\Bf_{c_{n-1}}$ is made unique by the argument in Section \ref{combinatorics}. See Figure~\ref{primitive}.

\begin{remNot}
To make our notations easier to follow, we drop the second subscripts whenever they are zero and do not 
create any confusion. Also, all objects on the dynamic planes of $\Bf_{c_{n-1}}$ and $\Bf_{\tilde{c}_{n-1}}$ 
(the ones after the straightenings) will be denoted by the boldface of the notations used for the 
corresponding objects on the dynamic planes of $f$ and $\tilde{f}$.
\end{remNot}

To define $\Omega_{n-1,j}$ and $h_{n-1,j}$, we need to consider the following three cases:
\begin{enumerate}
\item[$\mathscr{A}$.] $\rr^{n-1}f$ is primitively renormalizable;
\item[$\mathscr{B}$.] $\rr^{n-1}f$ is satellite renormalizable and $\rr^{n} f$ is primitively renormalizable;
\item[$\mathscr{C}$.] both $\rr^{n-1}f$ and $\rr^{n} f$ are satellite renormalizable.
\end{enumerate}

We introduce $\Omega_{n-1,j}$ and $h_{n-1,j}$ in each of the above cases. Then one consecutively applies these 
cases to construct all the domains $\Omega_{n-1,j}$ and the maps $h_{n-1,j}$. 
In what follows, to explain how we choose these cases consecutively, we associate a string of cases to $f$, 
depending only on its combinatorics. \label{cases}
More precisely, for a string of cases $\mathscr{A}^{m_1} \mathscr{B}^{m_2} \mathscr{C}^{m_3} \ldots$,
with non-negative integers $m_j$, obtained for $f$, we repeat Case $\mathscr{A}$, $m_1$ times then 
repeat Case $\mathscr{B}$, $m_2$ times and so on.

Given $f_c$, its renormalization on each level is of primitive or satellite type. 
Therefore, we can associate the string
\begin{equation}\label{word}
P\ldots PS\ldots SP\ldots
\end{equation}
in $P$ and $S$, where $P$ or $S$ in the $i$-th place means that the $i$-th renormalization of $f_c$ is of 
primitive or satellite type, respectively. 
Corresponding to any such string, we define a string of cases 
$\mathscr{A}^{m_1} \mathscr{B}^{m_2} \mathscr{C}^{m_3} \ldots$, with non-negative integers $m_j$, as follows.
Inductively, starting from the left, $P$ is replaced by $\mathscr{A}$, $SP$ by $\mathscr{B}$, and $SS$ by 
$\mathscr{C} S$. 
Repeating this process, we obtain the desired string of cases. 
This also introduces the sequence $n(k)$, for $k=1, 2, 3, \dots$ as follows. 
Given the string \eqref{word}, 
the sequence $n(k)$ is obtained from the sequence of natural numbers $1, 2, 3, \dots$ by removing all the 
integers $l$ for which there is an $S$ in the $(l-1)$-th place and a $P$ in the $l$-th place. 
That is, we skip the level of any primitive renormalization occurring after a satellite one.       
For example: 
\begin{align*}
\text{A string of renormalizations: }  &PSPPSSPSSSPP \dots  \\
\text{Its string of cases: }  &\mathscr{A} \mathscr{B} \mathscr{A} \mathscr{C} \mathscr{B} \mathscr{C} \mathscr{C}
\mathscr{B} \mathscr{A} \dots \\
\text{Its sequence: }  &1,2,4,5,6,8,9,10,12, \dots
\end{align*}

The following lemma shows that there are equipotentials of sufficiently high level $\eta(\eps)$ inside 
$S_{n-1}(W_{n-1,0})$ and $\widetilde{S}_{n-1}(\tilde{W}_{n-1,0})$ in the dynamic planes of the 
maps $\Bf_{c_{n-1}}$ and $\Bf_{\tilde{c}_{n-1}}$.
\begin{lem} \label{equipot}
For every $\eps>0$ and positive integer $d$, there exists  $\eta >0$ such that if $P_{c}:U'\ra U$ is a proper 
unicritical polynomial of degree $d$ with connected Julia set and $\mod (U\setminus U')\geq \eps$, 
then $U'$ contains equipotentials of level less than $\eta$.
\end{lem}
\begin{proof} The map $P_{c}$ on the complement of $K(P_{c})$ is conjugate to the map $P_{0}$ on the complement of 
the closed unit disk $D_1$ by the B\"{o}ttcher coordinate $B_c$. 
Since levels of the equipotentials are preserved under this map, and modulus is a conformal invariant, it is enough 
to prove the lemma for $P_{0}:V' \ra V$, with $V'$ compactly contained in $V$ and 
$\mod (V\setminus V')\geq \eps$. 
As $P_0:P^{-1}_0(V \setminus V') \ra (V\setminus V')$ is a covering of degree $d$, the modulus of the annulus 
$P^{-1}_0(V\setminus V')$ is at least $\eps/d$. 
Hence, $\mod (V' \setminus D_1)$ is at least $\eps/d$. 
By the Gr\"{o}tzsch problem in \cite{Ah66} (Section A in Chapter III) we conclude that $V'\setminus D_1$ 
must contain a round annulus $D_{\eta}\setminus D_1$.
\end{proof}
\medskip
\subsection*{Case $\mathscr{A}$:}
By considering equipotentials of level $\eta(\eps)$ contained in $S_{n-1}$ of $W_{n-1,0}$ and 
$\widetilde{S}_{n-1}$ of $\tilde{W}_{n-1,0}$, obtained in the previous lemma, and the external rays landing at 
the dividing fixed points $\alf_{n-1}$ and $\tilde{\alf}_{n-1}$ of the maps $\Bf_{c_{n-1}}$ and 
$\Bf_{\tilde{c}_{n-1}}$, we form the favorite nest of puzzle pieces \eqref{nest} introduced 
in Section~\ref{favnest}. 
\begin{figure}[ht]
\begin{center}
\includegraphics[scale=.9]{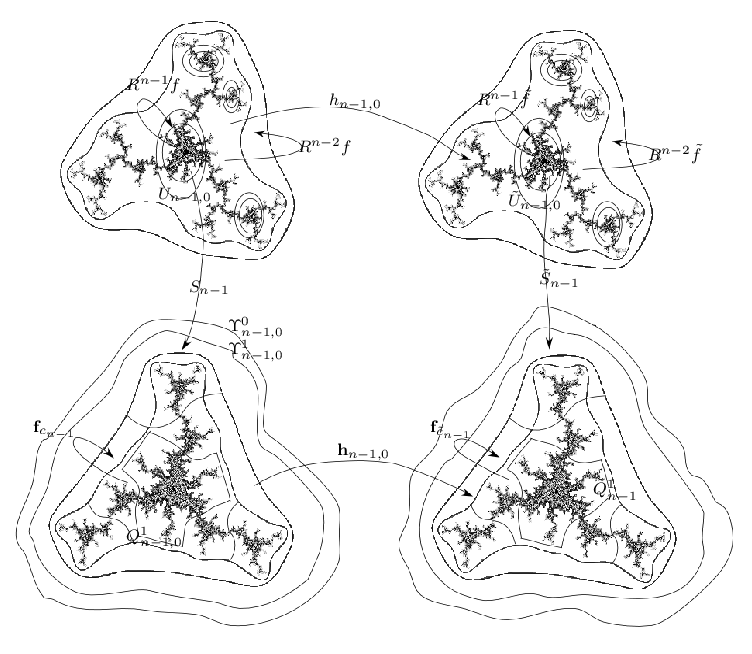}
\caption{Primitive case} \label{primitive}
\end{center}
\end{figure}

Let $Q^{\chi_{n}}_{n,0}:=Y^{q_{\chi_n}}_{n,0}$ and $P^{\chi_{n}}_{n,0}$ denote the last critical puzzle 
pieces  obtained in the nest \eqref{nest}, i.e.\ $\chi_n$ is the smallest positive integer with 
$\Bf_{c_{n-1}}^{t_{n}/t_{n-1}}:P^{\chi_{n}}_{n} \ra Q^{\chi_{n}}_{n}$ is a unicritical polynomial-like mapping of
degree $d$ with connected Julia set. Now, let 
\[\Bh_{n-1}=\Bh_{n-1,0}:\mathbb{C}\ra\mathbb{C},\] 
denote the $K_2$-q.c.\ pseudo-conjugacy one obtains from Theorem~\ref{pseudo-conj}. 
The distance between the parameters $c_{n-1}$ and $\tilde{c}_{n-1}$ with respect to the hyperbolic metric on 
the truncated primary wake containing $c_{n-1}$ and $\tilde{c}_{n-1}$, $W(\eta(\eps))$, is bounded by some 
$M$ depending only on the combinatorial class $\mathcal{SL}$. 
That is because $c_{n-1}$ and $\tilde{c}_{n-1}$ belong to a compact subset of $W(\eta(\eps))$ consisting of a 
finite number of truncated limbs. 
It has been proved in \cite{Ly97}, Theorem I, that $\mod (Q^1\setminus P^1)$ is 
uniformly bounded away from zero depending only on the combinatorial class $\mathcal{SL}$. 
The same proof, based on the continuous dependence of certain rays on the parameter and compactness of the 
class $\mathcal{SL}$, works for the unicritical polynomials as well. 
Therefore, by Proposition~\ref{unifrom-pseudo-conj}, $\Dil (\Bh_{n-1})$ is uniformly bounded by some 
constant $K_2$ depending only on $\eps$ and the class $\mathcal{SL}$. 

For $i=0,1,2,\dots,t_{n}/t_{n-1} -1$, denote the components of the sets $\Bf^{-i}_{c_{n-1}}(Q^{\chi_{n}}_{n,0})$ and 
$\Bf^{-i}_{c_{n-1}}(P^{\chi_{n}}_{n,0})$ containing 
\[\BJ_{1,i}:=\Bf_{c_{n-1}}^{t_n/t_{n-1}-i}(J(\Bf_{c_{n-1}}^{t_n/t_{n-1}}: P_n^{\chi_n}\ra Q_n^{\chi_n}),\] 
 by $Q^{\chi_{n}}_{n,i}$ and $P^{\chi_{n}}_{n,i}$, respectively. 
Note that $t_{n}/t_{n-1}$ is the period of the first renormalization of $\Bf_{c_{n-1}}$.

The polynomials $\Bf_{c_{n-1}}$ and $\Bf_{\tilde{c}_{n-1}}$ also satisfy our combinatorial and 
\textit{a priori} bounds assumptions. Therefore, there is a topological conjugacy $\Bpsi_{n-1}$ between 
them,  obtained from extending $B_{\tilde{c}_{n-1}}^{-1} \circ B_{c_{n-1}}$ onto $J(\Bf_{c_{n-1}})$.  

Now, we adjust $\Bh_{n-1}:Q^{\chi_{n}}_{n} \ra \widetilde{Q}^{\chi_{n}}_{n}$, using the dynamics 
of the maps 
\[\Bf_{c_{n-1}}^{t_{n}/t_{n-1}}:P^{\chi_{n}}_{n} \ra Q^{\chi_{n}}_{n}\quad \text{ and } \quad \Bf_{\tilde{c}_{n-1}}^{t_{n}/t_{n-1}} :\widetilde{P}^{\chi_{n}}_{n} \ra \widetilde{Q}^{\chi_{n}}_{n},\]
so that the equivariance property (for $\Bh_{n-1}$) holds on a larger set.
Let $A^0_n$ denote the closure of the annulus $Q^{\chi_{n}}_{n} \setminus P^{\chi_{n}}_{n}$, and $A^k_n$, 
for $k=1, 2, 3, \dots$, denote the component of $\Bf_{c_{n-1}}^{-k t_{n}/t_{n-1}}(A^0_n)$ around $J_{n,0}$. 
Then, we lift $\Bh_{n-1}$ via $\Bf_{c_{n-1}}^{t_n/t_{n-1}}:Q^{\chi_n}_{n}\ra \cc$ and 
$\Bf_{\tilde{c}_{n-1}}^{t_n/t_{n-1}}: \widetilde{Q}^{\chi_n}_n\ra \cc$ to obtain a $K_2$-q.c.\ 
homeomorphism $g:A^0_n\ra \widetilde{A}^0_n$ which is homotopic to $\Bpsi_{n-1}$ relative $\partial A^0_n$. 
That is because the external rays connecting $\partial P^{\chi_{n}}_{n}$ to $\partial Q^{\chi_{n}}_{n}$, 
partition $A^0_n$ into several topological disks, and  the two maps $\Bpsi_{n-1}$ and $g$ 
coincide on the boundaries of these topological disks.

As $\Bf_{c_{n-1}}^{k t_{n}/t_{n-1}}: A^k_n \ra A^0_n$ and 
$\Bf_{\tilde{c}_{n-1}}^{k t_{n}/t_{n-1}}: \widetilde{A}^k_n \ra \widetilde{A}^0_n$, for every $k\geq1$, 
are holomorphic unbranched coverings, $g$ can be lifted to a $K_2$-q.c.\ homeomorphism from $A^k_n$ to
$\widetilde{A}^k_n$. 
All these lifts are the identity in the B\"{o}ttcher coordinates on the boundaries of 
these annuli. Hence, they glue together to form a q.c.\ homeomorphism 
$g: Q^{\chi_{n}}_{n} \setminus \BJ_{1,0}\to \tilde{Q}^{\chi_{n}}_{n} \setminus \tilde{\BJ}_{1,0}$ which conjugate 
\[\Bf_{c_{n-1}}^{t_n/t_{n-1}}:P^{\chi_{n}}_{n} \setminus \BJ_{1,0}\ra Q^{\chi_{n}}_{n} \setminus \BJ_{1,0} 
\text{, to }\;  \Bf_{\tilde{c}_{n-1}}^{t_n/t_{n-1}}:\widetilde{P}^{\chi_{n}}_{n} \setminus \widetilde{\BJ}_{1,0}
\ra \widetilde{Q}^{\chi_{n}}_{n} \setminus \widetilde{\BJ}_{1,0}.\]

Finally, we would like to extend this map (as a homeomorphism) onto $\BJ_{1,0}$. 
This is a special case of a more general argument presented below.

Given a polynomial $f$ with connected filled Julia set $K(f)$, the \textit{rotation} of angle $\theta$ 
on $\cc\setminus K(f)$ is defined as the rotation of angle $\theta$ in the B\"{o}ttcher coordinate on 
$\cc\setminus K(f)$, that is, $B_c^{-1} (e^{\Bi\theta}\cdot B_c)$. By means of straightening, one can 
define rotations on the complement of the filled Julia set of a polynomial-like mapping. 
It is not canonical as it depends on the choice of the straightening. 
However, its effect on the landing points of external rays is canonical.

\begin{prop}\label{commute} 
Let $f: V_2 \ra V_1$ be a polynomial-like mapping with connected filled Julia set $K(f)$. 
If $\phi: V_1\setminus K(f)\ra V_1\setminus K(f)$ is a homeomorphism which commutes with $f$, then there exists 
a rotation of angle $2\pi j/(d-1)$ for some $j$, denoted by $\rho_j$, such that $\rho_j\circ\phi$ extends onto $K(f)$ as the identity.
\end{prop}
A proof of this proposition is given in the Appendix.

Applying the above proposition to ${\Bpsi}_{n-1}^{-1} \circ g$ with $V_1=Q^{\chi_{n}}_{n}$, $V_2=P^{\chi_{n}}_{n}$, 
and an external ray connecting $\partial Q^{\chi_{n}}_{n}$ to $\BJ_{1,0}$, we conclude that $g$ extends onto $\BJ_{1,0}$ 
as $\Bpsi_{n-1}$. 
Also, it follows from the proof of the above proposition that $g$ and $\Bpsi_{n-1}$ are homotopic relative 
$\BJ_{1,0}\cup \partial Q^{\chi_{n}}_{n}$. That is because the quadrilaterals obtained in the proof cut the 
puzzle piece $Q^{\chi_{n}}_{n}$ into infinite number of topological disks such that $g$ and $\Bpsi_{n-1}$ are equal on 
the boundaries of these topological disks.

Similarly, $\Bh_{n-1}$ can be adjusted on the other puzzle pieces $Q^{\chi_{n}}_{n,i}$, for $i=1$, $2$, $\dots$, 
$t_n/t_{n-1}$, so that it is homotopic to $\Bpsi_{n-1}$ restricted to $Q^{\chi_{n}}_{n,i}$, 
relative $\BJ_{1,i}\cup\partial Q^{\chi_{n}}_{n,i}$. We denote the map obtained from extending $\Bh_{n-1}$ onto little Julia 
sets $\BJ_{1,i}$ with the same notation $\Bh_{n-1}$.

Finally, we need to prepare $\Bh_{n-1}$ for the next step (case) of the process. It is stated in the following lemma.   
\begin{lem} \label{adjusted}
The $K_2$-q.c.\ homeomorphism $\Bh_{n-1}$ can be adjusted on a neighborhood of $\cup_i \BJ_{1,i}$, through a homotopy relative $\cup_i \BJ_{1,i}$, 
 to a q.c.\ homeomorphism $\Bh'_{n-1}$ which maps 
$\BV_{n,i}:= S_{n-1}(V_{n, i})$ onto $\widetilde{\BV}_{n,i}:=\widetilde{S}_{n-1}(\widetilde{V}_{n,i})$. 
Moreover, $\Dil (\Bh'_{n-1})$ is uniformly bounded by a constant $K_3$ depending only on $\eps$.
\end{lem}

\begin{proof}
The basic idea is to continuously move the domain $\Bh_{n-1}(\BV_{n,i})$ close enough to $J_{1,i}$, and then move it
back to $\widetilde{\BV}_{n,i}$ (simultaneously for all $i=0,1,\dots, t_n/t_{n-1}$). 
We will do this more precisely below. Let $\BU_{n,i}$ denote $S_{n-1}(U_{n,i})$ and $\widetilde{\BU}_{n,i}$ denote $\widetilde{S}_{n-1}(\widetilde{U}_{n,i})$.
 
The annuli $\Bh_{n-1}(\BV_{n,i})\setminus \widetilde{\BJ}_{1,i}$ and $\widetilde{\BV}_{n,i}\setminus \widetilde{\BJ}_{1,i}$ have moduli bigger than $\eps/dK_1K_2$, where $K_1:=\Dil (S_{n-1})$ and $K_2:=\Dil (\Bh_{n-1})$. 
Therefore, there exist topological disks $\widetilde{\BL}_{n,i}\supseteq \widetilde{\BJ}_{1,i}$, with 
smooth boundaries, and a constant $r>0$ satisfying the following properties 
\begin{itemize}
\item[--] $\widetilde{\BL}_{n,i} \subset \Bh_{n-1}(\BV_{n,i}) \cap \widetilde{\BV}_{n,i}$,
\item[--] $\mod (\widetilde{\BL}_{n,i}\setminus \widetilde{\BJ}_{1,i}) \geq r$,
\item[--] $\mod (\widetilde{\BV}_{n,i}\setminus \widetilde{\BL}_{n,i}) \geq \eps/2dK_1K_2$, and  
\item[--] $\mod (\Bh_{n-1}(\BV_{n,i})\setminus \widetilde{\BL}_{n,i}) \geq \eps/2dK_1K_2$.
\end{itemize}

Now, there are q.c.\ homeomorphisms 
\[\chi_i:\big(\Bh_{n-1}(\BU_{n,i}),\Bh_{n-1}(\BV_{n,i}),\widetilde{\BL}_{n,i}, \BJ_{1,i}\big)\ra \big(D_5, D_3,D_2,D_1\big),\]
with uniformly bounded dilatations. 
That is because the annuli 
\[\widetilde{\BL}_{n,i}\setminus \widetilde{\BJ}_{1,i}, \Bh_{n-1}(\BV_{n,i}) \setminus \widetilde{\BL}_{n,i}, \text{ and } \Bh_{n-1}(\BU_{n,i})\setminus\Bh_{n-1}(\BV_{n,i})\] 
have moduli uniformly bounded from above and away from zero independent of $n$ and~$i$. 

The homotopy \text{$g_t\colon\Dom (\Bh_{n-1})\ra \cc$}, for $t\in [0,1]$, is defined as 
\[\begin{cases} \Bh_{n-1}(z) &\text{ if $z\notin \bigcup_i \BV_{n,i}$}\\
                 \chi_i^{-1}\Big((\frac{-t}{3}\sin \frac{(|\chi_i \circ \Bh_{n-1}(z)|-1)\pi}{4}+1)\cdot \chi_i \circ \Bh_{n-1}(z)\Big) & \text{ if $z\in \BV_{n,i}.$} 
\end{cases}\]
It is straight to see that $g_0=\Bh_{n-1}$ on $\Dom \Bh_{n-1}$, $g_t$ is a well defined homeomorphism for every fixed $t$, and $g_t$ depends continuously on $t$ for every fixed $z$. For every $z\in \partial \BV_{n,i}$, at time $t=1$, we have  $g_1(z)=\chi_i^{-1}(\frac{2}{3}\cdot \chi_i\circ \Bh_{n-1}(z))\in \partial \widetilde{\BL}_{n,i}$. That is, $g_1$ maps $\partial V_{n,i}$ to $\partial \widetilde{\BL}_{n,i}$. 

For the other part, we consider q.c.\ homeomorphisms  
\[\Theta _i:\big(\widetilde{\BU}_{n,i},\widetilde{\BV}_{n,i},\widetilde{\BL}_{n,i},\widetilde{\BJ}_{1,i}\big)\ra\big(D_5, D_3,D_2,D_1\big),\]
and define $g_{t+1}:\Dom \Bh_{n-1}\ra \cc$, for $t\in [0,1]$, as      
\[\begin{cases} g_1(z) &\text{ if $z \notin g_1^{-1}(\bigcup_i \widetilde{\BU}_{n,i})$}\\
                 \Theta_i^{-1}\Big((\frac{t}{\sqrt 2}\sin \frac{(|\Theta_i \circ g_1(z)|-1)\pi}{4}+1)\cdot \Theta_i \circ g_1(z)\Big) &\text{ if $z\in g_1^{-1}(\widetilde{\BU}_{n,i}).$} 
   \end{cases}\]
The homotopy $g_t$ for $t\in[0,2]$ is the desired adjustment. The map $g_2:\Dom \Bh_{n-1}\ra \text{Range } \Bh_{n-1}$, is denoted by $\Bh'_{n-1}.$ 
\end{proof}

Let $\Delta_{n-1,0}$ denote the $S_{n-1}$-preimage of the domain bounded by $E^{\eta(\eps)}$ in the 
dynamic plane of $\Bf_{c_{n-1}}$. The domain $\Omega_{n-1,0}$ is defined as
\[\Delta_{n-1,0}\setminus \bigcup_{i=0}^{t_{n}/t_{n-1}} V_{n,i t_{n-1}}.\] 
The domains $\Delta_{n-1,i}$ and $\Omega_{n-1,i}$, for $i=1,2,\dots,t_{n-1}$, are defined as the pullback of $\Delta_{n-1,0}$ and $\Omega_{n-1,0}$, respectively, under $f^{-i}$ along the orbit of the critical point. 

Consider the map  
\begin{equation}\label{h'_{n-1}}
h_{n-1,0}:=\widetilde{S}_{n-1}^{-1} \circ \Bh'_{n-1} \circ S_{n-1}:\Delta_{n-1,0}\ra \widetilde{\Delta}_{n-1,0},
\end{equation}
and then, 
\begin{equation}\label{h'_{n-1,i}}
h_{n-1,i}:=\tilde{f}^{-i} \circ h_{n-1,0} \circ f^{i}:\Delta_{n-1,i}\ra \widetilde{\Delta}_{n-1, i},
\end{equation}
for the appropriate choice of the inverse branch of $f^i$. 
As these maps are compositions of two $K_1(\eps)$-q.c.\ \!\! and a $K_3(\eps)$-q.c.\ (and possibly some conformal  maps), they are q.c.\ with uniformly bounded dilatations. By our adjustment in Lemma~\ref{adjusted}, we know that $h_{n-1,i}$ maps $\Omega_{n-1,i}$ onto $\widetilde{\Omega}_{n-1,i}$. 

Finally, the annulus $V_{n-1,0}\setminus W_{n-1,0}$, with modulus bigger than $\eps/d$, encloses $\Omega_{n-1,0}$ and is contained in $V_{n-1, 0}$. This proves that $\Omega_{n-1,0}$ is well inside the disk $D_{n-1,0}:=V_{n-1,0}$. Similarly, the appropriate preimage of $V_{n-1,0}\setminus W_{n-1,0}$ under the conformal map $f^{i}$ introduces a definite annulus around $\Omega_{n-1,i}$ which is contained in $D_{n-1,i}:=V_{n-1,i}$. In this case, $D_{n,i}$ is  defined as $V_{n,i}$ which contains $J_{n,i}$ well inside itself.  
\medskip

\subsection*{Case $\mathscr{B}$:}
Here $\Bf_{c_{n-1}}$ is satellite renormalizable and its second renormalization is of primitive type. Let $\alf_{n-1}$ denote the dividing fixed point of $\Bf_{c_{n-1}}$, and $\alf_n\in \BJ_{1,0}:=J(\rr\Bf_{c_{n-1}})$ denote the dividing fixed point of $\rr\Bf_{c_{n-1}}$. By definition, the little Julia sets $\BJ_{1,i}$, $i=1,2,\dots, t_n/t_{n-1}$, of $\Bf_{c_{n-1}}$ touch each other at the $\alf_{n-1}$ fixed point. Here, 
$J(\rr^2 \Bf_{c_{n-1}})$, and its forward images under iterates of $\Bf_{c_{n-1}}$ can be arbitrarily close to $\alf_{n-1}$ (which is a non-dividing fixed point of $\rr \Bf_{c_{n-1}}$). Our idea is to skip the satellite level and start with the primitive one. This essentially imposes the $\mathcal{SL}$ condition on us. 

Consider an equipotential $E^{\eta(\eps)}$ contained in $S_{n-1}(W_{n-1,0})$, the external rays landing at $\alf_{n-1}$, and the external rays landing at the $\Bf_{c_{n-1}}$-orbit of $\alf_{n}$ (see Figure~\ref{satelite}).  Let us denote $\Bf_{c_{n-1}}^{t_n/t_{n-1}}$ by $g$, for the 
simplicity of the notation, throughout this case. 
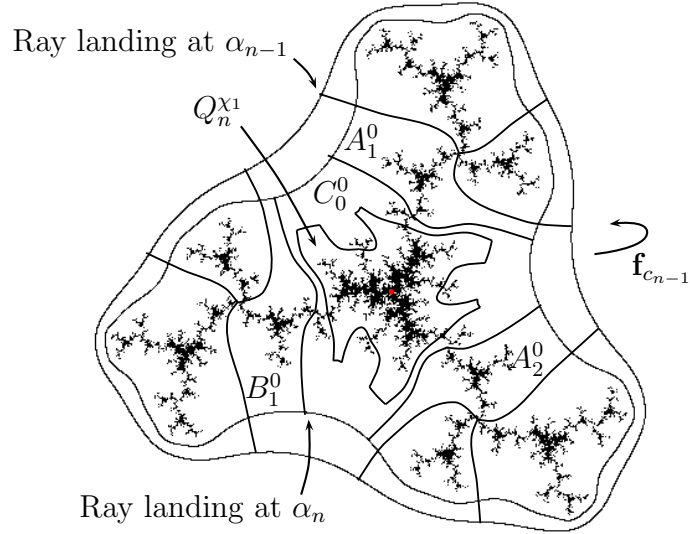
\begin{figure}[ht]
\begin{center}
  \begin{pspicture}(-4.28,-3.56)(3.74,4.44)
  \psset{xunit=1cm}
  \psset{yunit=1cm}
 %  \psframe(0,0)(8.5,7.5)
  \epsfxsize=7.6cm
  \rput(-.34,.36){\epsfbox{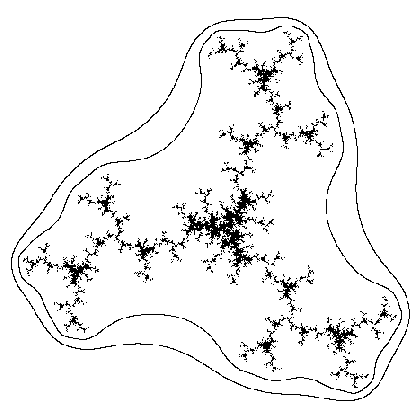}}
 % \psdots[linecolor=red,dotsize=2pt](0,0)
  \SpecialCoor  % at alpha fixed point
%the three rays landing at \alpha_{n-1}
  \pscurve[linewidth=.7pt](2.55;20)(2.0;28)(1.7;60)(2.05;64)
  \pscurve[linewidth=.7pt](3.3;51)(2.8;54)(2.2;59)(2.05;64)
  \pscurve[linewidth=.7pt](2.8;110)(2.4;95)(2.2;69)(2.05;64)
%preimage
  \pscurve[linewidth=.7pt](2.55;140)(2.0;148)(1.7;180)(2.05;184)
  \pscurve[linewidth=.7pt](3.3;171)(2.8;174)(2.2;179)(2.05;184)
  \pscurve[linewidth=.7pt](2.8;230)(2.4;215)(2.2;189)(2.05;184)
%preimage
  \pscurve[linewidth=.7pt](2.55;260)(2.0;268)(1.7;300)(2.05;304)
  \pscurve[linewidth=.7pt](3.3;291)(2.8;294)(2.2;299)(2.05;304)
  \pscurve[linewidth=.7pt](2.8;350)(2.4;335)(2.2;309)(2.05;304)
%rays landing at \alpha_n
  \pscurve[linewidth=.7pt,showpoints=false](2;20)(1.4;40)(1.;59)(1.05;75)
  \pscurve[linewidth=.7pt,showpoints=false](2;115)(1.25;83)(1.05;75)
%preimage
  \pscurve[linewidth=.7pt,showpoints=false](2;141)(1.4;160)(1.;179)(1.05;195)
  \pscurve[linewidth=.7pt,showpoints=false](2;235)(1.25;203)(1.05;195)
%preimage
  \pscurve[linewidth=.7pt,showpoints=false](2;261)(1.4;280)(1.;299)(1.05;315)
  \pscurve[linewidth=.7pt,showpoints=false](2;355)(1.25;323)(1.05;315)
%critical puzzle piece
  \pscurve[linewidth=.7pt,showpoints=false](1.15;355)(.62;350)(.71;7)   %1

  \pscurve[linewidth=.7pt,showpoints=false](1.3;17)(.9;20)(.71;7)   %2

  \pscurve[linewidth=.7pt,showpoints=false](1.3;17)(1.4;22)(1.5;27)   %3

  \pscurve[linewidth=.7pt,showpoints=false](1.5;27)(1.3;37)(.92;55)(.9;80)   %4

  \pscurve[linewidth=.7pt,showpoints=false](1.2;110)(1;100)(.9;80)      %5

  \pscurve[linewidth=.7pt,showpoints=false](1.15;355)(1.16;352)(1.2;350)      %6
%preimages 
  \pscurve[linewidth=.7pt,showpoints=false](1.15;115)(.62;110)(.71;127)   %1

  \pscurve[linewidth=.7pt,showpoints=false](1.3;137)(.9;140)(.71;127)   %2

  \pscurve[linewidth=.7pt,showpoints=false](1.3;137)(1.4;142)(1.5;147)   %3

  \pscurve[linewidth=.7pt,showpoints=false](1.5;147)(1.3;157)(.92;175)(.9;200)   %4

  \pscurve[linewidth=.7pt,showpoints=false](1.2;230)(1;220)(.9;200)      %5

  \pscurve[linewidth=.7pt,showpoints=false](1.15;115)(1.16;112)(1.2;110)      %6

%preimages
  \pscurve[linewidth=.7pt,showpoints=false](1.15;235)(.62;230)(.71;247)   %1

  \pscurve[linewidth=.7pt,showpoints=false](1.3;257)(.9;260)(.71;247)   %2

  \pscurve[linewidth=.7pt,showpoints=false](1.3;257)(1.4;262)(1.5;267)   %3

  \pscurve[linewidth=.7pt,showpoints=false](1.5;267)(1.3;277)(.92;295)(.9;320)   %4

  \pscurve[linewidth=.7pt,showpoints=false](1.2;350)(1;340)(.9;320)      %5

  \pscurve[linewidth=.7pt,showpoints=false](1.15;235)(1.16;232)(1.2;230)      %6

%Center  \psdots[linecolor=red,dotsize=2pt](2.05;64)(2.05;184)(2.05;304) 
  \pscurve{->}(-1.3,3.1)(-1.1,2.9)(-1,2.7)
  \rput(-3.2,3.3){Ray landing at $\alpha_{n-1}$}
  \pscurve{->}(-2,2.2)(-1.5,1.5)(-1,.6) 
  \rput(-2.3,2.4){$Q^{\chi_1}_n$}
  \rput(-.4,2){$A^0_1$}
  \rput(-.8,1.3){$C_0^0$}
  \rput(-1.7,-1.3){$B^0_1$}
  \rput(1.8,-.9){$A^0_2$}
  \pscurve{->}(-1.2,-2.7)(-1.1,-2.3)(-1.1,-1.7)
  \rput(-2.5,-2.9){Ray landing at $\alpha_n$}
  \pscurve{->}(2.7,.5)(3.4,.8)(2.9,.9)
  \rput(3.6,.3){$\mathbf{f}_{c_{n-1}}$}
% \psgrid[gridcolor=gray]
 \end{pspicture}

 \caption[An infinitely renormalizable Julia set]{The figure shows an infinitely renormalizable Julia set. The first renormalization is of satellite type and the second one is of primitive type. The puzzle piece $Q^{\chi_1}_n$ at the center is the first puzzle piece in the favorite nest.}
\label{satelite}
\end{center}
\end{figure}

Let $Y^1_0$, as before, denote the critical puzzle piece of level one bounded by $E^{\eta(\eps)}$, the external rays landing at $\alf_{n-1}$, and the external rays landing at the points in $\Bf_{c_{n-1}}^{-1}(\alf_{n-1})$. 
The external rays landing at $\alf_n$ and their $g$-preimage, cut the puzzle piece $Y^1_0$ into finitely many pieces. Let us denote the one containing the critical point by $C^0_0$, the non-critical ones with $\alf_n$ on their boundary by $B^0_i$, and the rest of them by $A^0_j$ (these ones have an $\omega \alf_n$ on their boundary, with  
$\omega^d=1$ but $\omega\neq 1$). 

The $g$-preimage of $Y_0^1$ along the post-critical set is contained in itself. As all the processes
of making the modified principal nest and the pseudo-conjugacy in Theorem \ref{pseudo-conj} are based on 
pullback arguments, the same ideas, which is explained briefly below, work here as well. 
The only difference is that we do not have any equipotential for the second renormalization of $\Bf_{c_{n-1}}$. 
However, certain external rays and part of $E^{\eta(\eps)}$ can play the role of an equipotential for 
$\rr^2 (\Bf_{c_{n-1}})$.

By definition of satellite and primitive renormalizability, $g^n(0)$ belongs to $Y_0^1$, for $n \geq 0$, and there is a smallest positive integer $t$ with $g^t(0)\in A^0_1$ (by rearranging the indices if required). 
Pulling $A^0_1$ back under $g^{t}$, along $0,g(0), \dots ,g^t(0)$, we obtain a puzzle piece $Q^{\chi_1}_n\ni 0$, such that $C^0_0\setminus Q^{\chi_1}_n$ is a non-degenerate annulus. That is because $C^0_0$ is bounded by the external rays landing at $\alf_n$ and their $g$-preimage. Therefore, if $Q^{\chi_1}_n$ intersects $\partial C^0_0$ at some point on the rays, orbit of this intersection under $g^k$, for $k\geq 1$, stays on the rays landing at $\alf_n$. This implies that image of $Q^{\chi_1}_n$ can never be $A^0_1$. Also, they do not intersect at equipotentials with  different levels. 

Now, let $m>t$ be the smallest positive integer with $g^{m}(0)\in Q^{\chi_1}_n$. Pulling $Q^{\chi_1}_n$ back under $g^{m}$, along the critical orbit, we obtain $P^{\chi_1}_n$. The map $g^{m}$ is a unicritical degree $d$ branched covering from $P^{\chi_1}_n$ onto $Q^{\chi_1}_n$. This introduces the first two pieces in the  favorite nest. The rest of the process to form the whole favorite nest is the same as in Section~\ref{favnest}. 

Consider the maps $\Bf_{c_{n-1}}^{t_n/t_{n-1}}:Y_0^1\ra \Bf_{c_{n-1}}^{t_n/t_{n-1}}(Y_0^1)$, and the corresponding tilde one. One applies Theorem~\ref{pseudo-conj} to these maps, using the favorite nests introduced in the above paragraph, to obtain a q.c.\ pseudo-conjugacy 
\[\Bh_{n-1}:\Bf_{c_{n-1}}^{t_n/t_{n-1}}(Y_0^1)\ra \Bf_{\tilde{c}_{n-1}}^{t_n/t_{n-1}}(\widetilde{Y}_0^1),\]
up to level of $Q^{\chi_n}_n$.
The equipotential $E^{\eta(\eps)}$, the external rays landing at $\alf_{n-1}$, and the external rays landing at the $\Bf_{c_{n-1}}$-orbit of $\alf_{n}$ depend holomorphically on the parameter within the secondary wake $W(\eta)$ containing the parameter $c_{n-1}$. The hyperbolic distance between $c_{n-1}$ and $\tilde{c}_{n-1}$ within one of the finitely many secondary wakes $W(\eta)$ is uniformly bounded in terms of the combinatorial class $\mathcal{SL}$.
Also, $\mod (Q^{\chi_1}_n\setminus P^{\chi_1}_n)$ is bounded away from zero for the parameters in these secondary limbs. Thus, by Proposition~\ref{unifrom-pseudo-conj}, $\Dil (\Bh_{n-1})$ depends only on the \textit{a priori} bounds $\eps$ and the combinatorial class $\mathcal{SL}$.  

As $\Bf_{c_{n-1}}^j:\Bf_{c_{n-1}}^{t_n/t_{n-1}-j}(Y_0^1)\ra \Bf_{c_{n-1}}^{t_n/t_{n-1}}(Y_0^1)$, for $j=1,2,\dots,t_n/t_{n-1}-1$, is univalent, we can lift $\Bh_{n-1}$ onto other puzzle pieces as 
\[\Bh_{n-1}:=\Bf_{\tilde{c}_{n-1}}^{-j} \circ \Bh_{n-1} \circ \Bf_{c_{n-1}}^j:\Bf_{c_{n-1}}^{t_n/t_{n-1}-j}(Y_0^1)\ra \Bf_{\tilde{c}_{n-1}}^{t_n/t_{n-1}-j}(\widetilde{Y}_0^1)\] 
for these values of $j$. Since all these maps match the B\"ottcher marking, they glue together to build a q.c.\ homeomorphism from a neighborhood of $J(\Bf_{c_{n-1}})$ to a neighborhood of $J(\Bf_{\tilde{c}_{n-1}})$. Then, one extends this map, as the identity in the B\"ottcher coordinates, to a q.c.\ homeomorphism from the domain bounded 
by $E^{\eta(\eps)}$ to the domain bounded by $\widetilde{E}^{\eta(\eps)}$.    

Finally, by Proposition \ref{commute} and Lemma~\ref{adjusted}, we adjust $\Bh_{n-1}$ to obtain a
q.c.\ homeomorphism $\Bh'_{n-1}$ that satisfies 
\[\Bh'_{n-1}(\Bf_{c_{n-1}}^i(J(\rr^2 (\Bf_{c_{n-1}}))))=\Bf_{\tilde{c}_{n-1}}^i(J(\rr^2 (\Bf_{\tilde{c}_{n-1}}))),\]
and 
\[\Bh'_{n-1}(S_{n-1}(V_{n+1,i t_{n-1}}))\!=\!\widetilde{S}_{n-1}(\widetilde{V}_{n+1, it_{n-1}}),\]
for  $i=0, 1,2,\dots, t_{n+1}/t_{n-1}-1.$ 

Now, $\Delta_{n-1,0}$ is defined as $S_{n-1}$-pullback of the domain bounded by $E^{\eta(\eps)}$. 
The domain $\Omega_{n-1, 0}$ is
\[\Delta_{n-1,0}\setminus \bigcup_{i=0}^{t_{n+1}/t_{n-1}-1} V_{n+1,it_{n-1}}.\] 
The regions $\Delta_{n-1,i}$ and $\Omega_{n-1,i}$, for $i=1,2,3,\dots,t_{n-1}$,  are the appropriate pullbacks of $\Delta_{n-1,0}$ and $\Omega_{n-1, 0}$ under $f^i$, respectively.
Like previous case, $h_{n-1,i}$ is defined as in Equation \ref{h'_{n-1}} or \ref{h'_{n-1,i}} and satisfies 
\[h_{n-1,i}(\Omega_{n-1,i})=\widetilde{\Omega}_{n-1,i}, \text{ for } i=1,2,\dots, t_{n-1}.\]
For the same reason as in Case $\mathscr{A}$, $\Omega_{n-1,i}$ is well inside $D_{n-1,i}:=V_{n-1,i}$, and
$J_{n+1,i}$ is well inside $D_{n+1,i}:=V_{n+1,i}$.
\smallskip

\subsection*{Case $\mathscr{C}$:}
Here, $J(\rr(\Bf_{c_{n-1}}))$ is denoted by $\BJ_{1,0}$, and $\Bf^{t_n/t_{n-1}-i}_{c_{n-1}}(\BJ_{1,0})$, 
for $i=1,2,\dots$, $\frac{t_n}{t_{n-1}}-1$, is denoted by $\BJ_{1,i}$. 
These little Julia sets touch each other at the dividing fixed point $\alf_{n-1}$ of $\Bf_{c_{n-1}}$. Note that $\alf_{n-1}$
is one of the non-dividing fixed points of $\rr\Bf_{c_{n-1}}$. The union of these little Julia sets is called the  \textit{Julia bouquet} and is denoted by $\BB_{1,0}$. 
\begin{figure}[ht]
\begin{center}
\includegraphics[scale=.5]{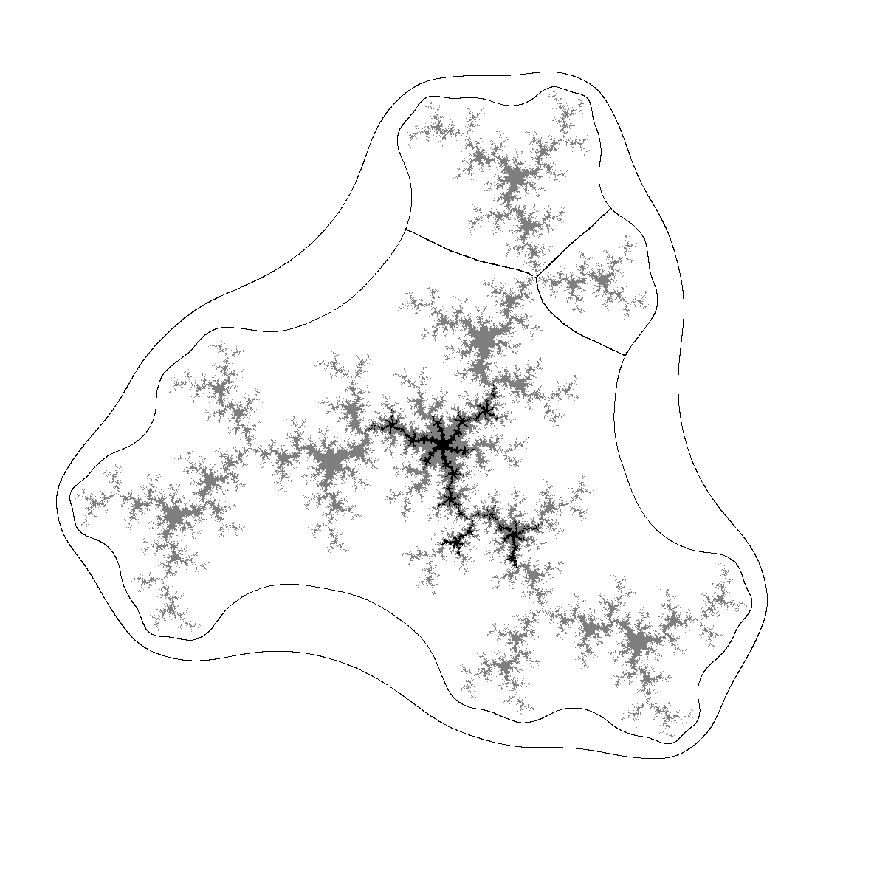}
\caption[A twice satellite renormalizable Julia set]{A twice satellite renormalizable Julia set drawn in grey. The dark part is the Julia 
Bouquet $\BB_{2,0}$.} \label{satsat}
\end{center}
\end{figure}

Similarly, $J(\rr^2(\Bf_{c_{n-1}}))$, and its forward iterates under $\Bf_{c_{n-1}}$ are denoted by $\BJ_{2,i}$, for $i=0,1,\dots,(t_{n+1}/t_{n-1})-1$,(i.e.\ $\Bf_{c_{n-1}}^i(\BJ_{2,i})=\BJ_{2,0}$). 
The connected components of the union of these little Julia sets are called the \textit{bouquets} $\BB_{2,j}$, $j=0,1,\dots,t_n/t_{n-1}-1$, (here, $\Bf_{c_{n-1}}^j(\BB_{2,j})=\BB_{2,0}$). 
In other words, 
\[\BB_{2,j}:=\bigcup_{k=0}^{\frac{t_{n+1}}{t_{n-1}}-1}\BJ_{2, k\frac{t_n}{t_{n-1}}+j}.\] 
Each $\BB_{2,j}$ consists of $t_{n+1}/t_{n}$ little Julia sets touching each other at one of
their non-dividing fixed points. As usual, $\BB_{2,0}$ denotes the bouquet
containing the critical point. See Figure \ref{satsat}.

By an equipotential $E^{\eta(\eps)}\subset S_{n-1}(W_{n-1,0})$, and the external rays landing at $\alf_{n-1}$, we form the puzzle pieces of level zero. Recall that $Y^0_0$ denotes the critical puzzle piece of level zero. 
The following lemma shows that the bouquets are well apart from each other. See Lemma~$10.5$ in \cite{Ly97} for a similar statement.  
\begin{lem}\label{bouquet}
For the twice satellite renormalizable parameters in a finite number of truncated secondary limbs, 
$\mod (Y^0_0 \setminus \BB_{2,0})$ is uniformly bounded from above and away from zero. 
\end{lem}
Let $X$ and $Y$ be non-empty compact subsets of $\cc$. The \textit{Hausdorff distance} between $X$ and $Y$ is defined as 
\[d_H(X,Y):= \inf \{\eps\in \mathbb{R} :Y\subset B_\eps(X)\text{, and } X\subset B_\eps(Y)\},\]
where $B_\eps(X)$ denotes the $\eps$ neighborhood of $X$ in the Euclidean metric. The space of all non-empty compact subsets of $\cc$ endowed with this metric is a complete metric space. 

We say that a family of simply connected domains $U_{\lambda}$, parametrized on a topological disk, depends continuously on $\lambda$, if there exist choices of uniformizations $\psi_{\lambda}:D_1\ra U_{\lambda}$ continuous in both variables.
We say that a family of polynomial-like mappings $(P_{\lambda}\!:\!V_{\lambda} \ra U_{\lambda},\lambda \!\in\! \Lambda)$, 
parametrized on a topological disk $\Lambda$, depends continuously on $\lambda$, if 
\begin{itemize}
\item[--] $U_{\lambda}$ is a continuous family in $\cc$, and
\item[--] for every fixed $z\in\cc$, the map $P_{\lambda}(z)$, restricted to the values of $\lambda$ it is defined, 
depends continuously on $\lambda$.
\end{itemize}

\begin{prop}
Let $(P_{\lambda}:V_{\lambda} \ra U_{\lambda},\lambda \in \Lambda)$ be a continuous
family of polynomial-like mappings with connected filled Julia sets $K_{\lambda}$. Then for every 
$\lambda\in \Lambda$ and every $\eps>0$, there exists $\eta>0$ such that if $\vert\lambda'-\lambda\vert<\eta$, 
then $K_{\lambda'}\subseteq B_\eps(K_\lambda)$.
\end{prop}

\begin{proof}
Assume that $z\notin B_\eps(K_{\lambda})$. If $z \notin V_\lambda$, then by continuous dependence of $V_\lambda$ on $\lambda$, $z\notin K_{\lambda'}$ for $\lambda'$ sufficiently close to $\lambda$. 
If $z\in V_\lambda$, then there exists a
positive integer $l$ with $P_{\lambda}^l(z) \in U_{\lambda}
\setminus V_{\lambda}$. As $P_{\lambda'}:V_{\lambda'}\ra U_{\lambda'}$ converges to
$P_{\lambda}:V_\lambda\ra U_\lambda$, when $\lambda'\ra\lambda$, at least one of $P_{\lambda'}^l(z)$ or $P_{\lambda'}^{l+1}(z)$ belongs to $U_{\lambda'}
\setminus V_{\lambda'}$, for $\lambda'$ sufficiently close to $\lambda$. 
Therefore, $z\notin K_{\lambda'}$.
\end{proof}
\begin{proof}[Proof of Lemma~\ref{bouquet}]
Let $c_{n-1}$ be a twice satellite renormalizable parameter in a given truncated secondary limb. Consider the external rays landing at the dividing fixed point $\alf_n$ of $\rr \Bf_{c_{n-1}}$ and their preimages under 
$\rr \Bf_{c_{n-1}}$. Let $X_0^0$ denote the critical puzzle piece obtained from these rays. 
As $\Bf_{c_{n-1}}$ is twice satellite renormalizable, $\rr^2 \Bf_{c_{n-1}}:X_0^0\ra \cc$ is a proper branched covering over its image. One can consider a continuous thickening of $X_0^0$, mentioned in Section \ref{thickening}, to form a continuous family of polynomial-like mappings parametrized over this truncated limb. 
The little filled Julia set of this map is denoted by $\BK_{2,0}$, and the filled bouquet $\BB_{2,0}$ is the connected component of 
\[\bigcup_{i=0}^{t_{n+1}/t_{n-1}-1} \Bf_{c_{n-1}}(\BK_{2,0})\]
containing the critical point.     

For a twice satellite renormalizable $c_{n-1}$ in a finite number of truncated secondary limbs, the filled bouquet $\BB_{2,0}$ is defined an is a union of a finite number of little filled Julia sets. Also $\BB_{2,0}$ is contained well inside the interior of $Y_0^0$ for $c_{n-1}$ in the closure of these parameter values. Therefore, by the above proposition, $\mod(Y_0^0\setminus \BB_{2,j})$ is uniformly bounded away 
from zero.

To see that these moduli are uniformly bounded from above, one needs only to observe that $\alf_{n}$ and $0$ belong to $\BB_{2,0}$ and are distinct for these parameters.  
\end{proof}

By Lemma~\ref{bouquet}, there are simply connected domains $\BL'_{n}\subseteq \BL_{n}$ and $\widetilde{\BL}'_{n}\subseteq \widetilde{\BL}_{n}$, with smooth boundaries, such that the moduli of the annuli 
\begin{equation}\label{annuli}
\begin{matrix}
Y_0^0 \setminus \BL_{n}, & \BL_{n} \setminus \BL'_{n}, & \BL'_{n} \setminus \BB_{2,0}\\
\widetilde{Y}_0^0 \setminus \widetilde{\BL}_{n}, & \widetilde{\BL}_{n} \setminus \widetilde{\BL}'_{n}, & \widetilde{\BL}'_{n} \setminus \widetilde{\BB}_{2,0}
\end{matrix} 
\end{equation}
are uniformly bounded from above and away from zero by some constants depending only on the combinatorial class $\mathcal{SL}$. It follows that the ratios 
\[\frac{\mod(Y_0^0 \setminus \BL_{n})}{\mod(\widetilde{Y}_0^0 \setminus \widetilde{\BL}_{n})},
\quad
\frac{\mod(\BL_{n}\setminus\BL'_{n})}{\mod(\widetilde{\BL}_{n}\setminus\widetilde{\BL}'_{n})},\quad
\frac{\mod(\BL'_{n}\setminus\BB_{2,0})}{\mod(\widetilde{\BL}'_{n} \setminus \widetilde{\BB}_{2,0})}\] 
are also uniformly bounded from above and away from zero independent of $n$.

Above data implies that there exists a q.c.\ homeomorphism 
\[\Bh_{n-1}:Y_0^0 \setminus \BL_{n}\ra \widetilde{Y}_0^0 \setminus \widetilde{\BL}_{n},\] 
with uniformly bounded dilatation, which matches the B\"ottcher marking on $\partial Y_0^0$ (See Lemma  \ref{adjust}). Then, we lift $\Bh_{n-1}$ via $\Bf_{c_{n-1}}^{-i}$ and $\Bf_{\tilde{c}_{n-1}}^{-i}$ to extend $\Bh_{n-1}$ to q.c.\ homeomorphisms 
\[\Bh_{n-1}:\Bf_{c_{n-1}}^{-i}(Y_0^0 \setminus \BL_{n})\ra \Bf_{\tilde{c}_{n-1}}^{-i}(\widetilde{Y}_0^0 \setminus \widetilde{\BL}_{n}), \text{ for } i=1,2,\dots,\frac{t_n}{t_{n-1}}-1.\]
The domain of each map above is a puzzle piece $Y^0_j$ (with $j=t_n/t_{n-1}-i$) cut off by $E^{\eta /d^i}$ and the appropriate component of $\Bf_{c_{n-1}}^{-i}(\partial\BL_{n})$. As all these maps match the B\"ottcher marking on  $\partial Y^0_j$, they can be glued together. Finally, one extends this map onto unbounded components of
$\mathbb{C} \setminus \Dom \Bh_{n-1}$ as the identity in the B\"ottcher coordinates. We denote this extended map with the same notation $\Bh_{n-1}$. As it is lifted under and extended by holomorphic maps, there is a uniform bound on its dilatation.  

Let $\Delta_{n-1,0}$ be the $S_{n-1}$-preimage of the domain inside $E^{\eta(\eps)}$, and $L_{n,i}$, for $i=1,2,\dots, \frac{t_n}{t_{n-1}}-1$, be the component of $S_{n-1}^{-1}(\Bf_{c_{n-1}}^{-i}(\BL_{n}))$ containing $J_{n-1,i}\cap \pc(f)$.
Then, define the multiply connected regions  
\[\Omega_{n-1, 0}:=\Delta_{n-1, 0}\setminus \bigcup_{i=0}^{t_n/t_{n-1}-1} L_{n,it_{n-1}}.\]
Like before, $\Delta_{n-1,i}$ and $\Omega_{n-1,i}$ are defined as the appropriate $f^{i}$-preimage of $\Delta_{n-1,0}$ and $\Omega_{n-1,0}$ intersecting $J_{n-1,i}$, respectively.
  
We have 
\[h_{n-1,0}:=\widetilde{S}_{n-1}^{-1}\circ \Bh_{n-1} \circ S_{n-1}:\Delta_{n-1,0}\ra \widetilde{\Delta}_{n-1,0},\] \[\text{ with } h_{n-1,0}(\Omega_{n-1,0})=\widetilde{\Omega}_{n-1,0}.\]
Also, 
\[h_{n-1,i}:=\tilde{f}^{-i}\circ h_{n-1,0} \circ f^i:\Delta_{n-1,i}\ra \widetilde{\Delta}_{n-1,i},\] 
\[\text{ with } h_{n-1,i}(\Omega_{n-1,i})=\widetilde{\Omega}_{n-1,i}.\]
As $E^{\eta(\eps)}$ is contained in $S_{n-1}(W_{n-1,0})$, $\Omega_{n-1,0}$ is contained in $W_{n-1,0}$. Therefore, $\Omega_{n-1,0}$ is well inside $D_{n-1,0}:=V_{n-1,0}$. The conformal invariance of modulus implies that the other domains $\Omega_{n-1,i}$ are well inside $D_{n-1,i}:=V_{n-1,i}$ as well. This completes the construction in Case $\mathscr{C}$.
\bigskip

To fit together the multiply connected domains $\Omega_{n(k),i}$ and the q.c.\ homeomorphisms $h_{n(k),i}:\Omega_{n(k),i}\ra \widetilde {{\Omega}}_{n(k),i}$, we follow the string of cases introduced before 
Lemma~\ref{equipot}. 
In Cases $\mathscr{A}$ and $\mathscr{B}$, we have adjusted $\Bh_{n-1}$, in Lemma~\ref{adjusted}, such that it sends $\partial V_{n,0}$ to $\partial \widetilde{V}_{n,0}$. Therefore, if any of the three cases follows a Case $\mathscr{A}$ or $\mathscr{B}$, we consider $\rr^{n}f:W_{n,0} \ra V_{n,0}$ and straighten it with these choices of domains (instead of considering $\rr^{n}f: V_{n,0} \ra U_{n,0}$). If the construction on some level $n$ follows a Case $\mathscr{C}$ (level $n-1$ belongs to $\mathscr{C}$), the set $\Delta_{n,0}$ introduced on level $n$ is not contained in the hole $D_{n,0}=L_{n,0}$ obtained on level $n-1$ in  Case $\mathscr{C}$. Note that here the $n$-th renormalization is of satellite type. The following paragraph explains the adjustment needed here.

As the annulus $\BDelta_{n,0}\setminus J(\Bf_{c_n})$ has definite modulus  in terms of $\eps$, the annulus $\BDelta_{n,0}\setminus \BB_{1,0}(\supset \BDelta_{n,0}\setminus J(\Bf_{c_n}))$, where $\BB_{1,0}:=\cup_{i=0}^{\infty} \Bf_{c_n}^i(J(\rr \Bf_{c_n}))$, also has definite modulus in terms of $\eps$. By quasi-invariance of modulus, the annulus obtained from 
\[S_n(L'_{n,0}\cap \Dom S_n)\setminus \BB_{1,0}\]
must have definite modulus in terms of $\eps$. Now, let $\BE_n$ be a topological disk contained in 
\[S_n(L'_{n,0}\cap \Dom S_n)\cap \BDelta_{n,0}\] 
with $\mod (\BE_n\setminus \BB_{1,0})$ bigger than a constant in terms of $\eps$ (similarly define the  corresponding tilde one). By a similar argument as in Lemma~\ref{adjusted} we can adjust $\Bh_n$ through a homotopy to obtain a map $\Bh'_{n}:\BE_n \ra \widetilde{\BE}_{n}$. Now, in this situation $\BDelta_{n,0}$ is replaced by $\mathbf{E}_n$, $\Delta_{n,0}$ by $S_n^{-1}(\BE_n)$, and $h_n$ by $\widetilde{S}_n^{-1} \circ \Bh'_n  \circ S_n$. The annulus $L_{n,0}\setminus L'_{n,0}$ provides the definite space around $\Omega_{n,0}$ in $L_{n,0}$.

In the following two sections, we will denote the holes of $\Omega_{n(k),i}$ by $\mathbb{V}_{n(k)+1,j}$. 
That is, $\mathbb{V}_{n(k)+1,j}$ is $V_{n(k)+1,j}$, if step $n(k)$ belongs to Case
$\mathscr{A}$ or $\mathscr{B}$, and $\mathbb{V}_{n(k)+1,j}$ is
$S^{-1}_{n(k)}(\BL_{n(k),j})$ if step $n(k)$ belongs to Case $\mathscr{C}$.

\subsection{The gluing maps $g_{n(k),i}$}
In this section we build q.c.\ homeomorphisms 
\[g_{n(k),i}:\mathbb{V}_{n(k),i}\setminus \Delta_{n(k),i}\ra \widetilde{\mathbb{V}}_{n(k),i}\setminus \widetilde{\Delta}_{n(k),i},\]
with uniformly bounded dilatations, needed to glue the maps $h_{n(k), i}$ together.  
Every $g_{n(k),i}$ must be identical with $h_{n(k-1),i}$ on $\partial \mathbb{V}_{n(k),i}$, and with $h_{n(k),i}$ on $\partial \Delta_{n(k),i}$ (which is the outer boundary of $\Omega_{n(k),i}$). Then, gluing all these maps $g_{n(k),i}$ and $h_{n(k),i}$ together produces a q.c.\ homeomorphism denoted by $H$. 
In what follows, for the simplicity of the notation, we use index $n$ instead of $n(k)$, and assume that $n$ runs over the  subsequence $\langle n(k)\rangle_{k=1}^\infty$. So, by ``for all $n$'' we mean ``for all $n(k)$''.

Like previous steps, first we build $g_{n,0}$, for all $n$, and then lift them via $f^{i}$ and $\tilde{f}^{i}$ to
obtain $g_{n,i}$, for $i=1, 2, \dots, t_n$. Definitions of the maps $h_{n,i}$ as well as the domains $\mathbb{V}_{n,i}$ and $\Omega_{n,i}$ imply that these maps also glue together on the boundaries of their domains of definition. Again, we drop the second subscript if it is zero, i.e., $g_n$ denotes the map $g_{n,0}$. 

To build a q.c.\ homeomorphism from an annulus to an annulus with given boundary conditions, there are $\mathbb{Z}$ possible choices for the number of the ``twists'' one may make. To have a uniform bound on the dilatation of such a map, not only must the two annuli have comparable moduli uniformly bounded away from zero but also the number of twists must be uniformly bounded as well. Note that the homotopy class of the final map $H$ depends on the choices of these twists.

In this section, we show that the annuli $\mathbb{V}_{n,0}\setminus \Delta_{n,0}$ and
$\widetilde{\mathbb{V}}_{n,0}\setminus \widetilde{\Delta}_{n,0}$ have comparable moduli. 
In the next section we prescribe the correct number of twists needed to obtain a Thurston conjugacy.

\begin{lem}\label{gluspace}
There exists a constant $M'$ depending only on $\eps$ such that for every $n\geq 1$, 
\[\frac{1}{M'}\leq \frac{\mod (\widetilde{\mathbb{V}}_{n,0}\setminus \widetilde{\Delta}_{n,0})}{\mod (\mathbb{V}_{n, 0}\setminus \Delta_{n, 0})} \leq M'.\] 
\end{lem}

\begin{proof}
If level $n$ follows one of Cases $\mathscr{A}$ or $\mathscr{B}$, then 
\[\mod (\mathbb{V}_{n,0}\setminus \Delta_{n,0})\leq \mod (V_{n,0}\setminus J_{n,0})\leq \eta,\] 
by \eqref{const-eta}. If level $n$ follows a 
Case $\mathscr{C}$ then 
\[\mod (\mathbb{V}_{n,0}\setminus \Delta_{n,0})\leq \mod S_n^{-1}(\BL_{n,0}\setminus \BB_{1,0}) \leq M'',\] 
for some constant $M''$ by the proof of Lemma~\ref{bouquet}  (Here $\BB_{1,0}$ is the unique bouquet of $\Bf_{c_n}$). 

Similarly, $\mod(\mathbb{V}_{n, 0}\setminus \Delta_{n,0})$ is bigger than $\eps$, or some constant depending on $\eps$, depending on whether level $n$ follows a Case $\mathscr{A}$, $\mathscr{B}$, or $\mathscr{C}$. 
Hence, $\mod(\mathbb{V}_{n,0}\setminus \Delta_{n,0})$ and $\mod(\widetilde{\mathbb{V}}_{n,0}\setminus \widetilde{\Delta}_{n,0})$ are pinched between two constants depending only on $\eps$. This implies the lemma.   
\end{proof}

Let $A(r)$ denote the round annulus $D_r\setminus D_1$, for $r>1$, and assume $\gamma:[0,1]\ra A(r)$
is a curve parametrized in the polar coordinate as $\gamma(t)=(r(t), \theta(t))$, for $t\in [0,1]$, with $r(t)$ and $\theta(t)$ continuous functions from $[0,1]$ to $\mathbb{R}$.
The \textit{wrapping number} of $\gamma$, denoted by $\omega(\gamma)$, is defined as $\theta(1)-\theta(0)$.
Given a curve $\gamma: [0,1] \ra U$, where $U$ is an annulus, with $\gamma(0)$ on the
inner boundary of $U$ (corresponding to the bounded component of $\cc\setminus U$) and $\gamma(1)$ on the outer boundary of $U$ (corresponding to the unbounded one), we define the wrapping number of $\gamma$ in $U$ as $\omega(\gamma):=\omega (\phi\circ \gamma)$,
where $\phi$ is a uniformization of $U$ by a round annulus. Note that $\omega(\gamma)$ is
invariant under the automorphism group of $U$. So, it is independent of the choice of the uniformization. 
In addition, just like winding number, it is constant over the homotopy class of all curves with the same boundary points.

\begin{prop} \label{spiral}
Given fixed constants $K\geq 1$ and $r>1$, there exists a constant $N$ such that for every $K$-q.c.\ homeomorphism $\psi:A(r) \ra A(r')$, the wrapping number of the curve $\psi(t), t\in[1,r]$, belongs to the interval $[-N,N]$.
\end{prop}

\begin{proof}
This follows from the compactness of the class of $K$-q.c.\ homeomorphisms from $A(r)$ to some $A(r')$.  
\end{proof}

In the following lemma let $\theta$ be a branch of the argument defined on $\mathbb{C}$ minus an straight ray from $0$ to infinity.
\begin{lem}\label{adjust}
Fix round annuli $A(r)$, $A(r')$, positive constants $\delta$, $K_1$ and $K_2$, as well as an integer $k$ with 
\[\mod A(r')/K_1 \leq \mod A(r)\leq K_1 \mod A(r'),\text{ and } \mod A(r)\geq \delta.\] 
There exists a constant $K$ depending only on $K_1$, $K_2$, $k$, and $\delta$, such that if homeomorphisms 
\[ h_1:\partial D_r \ra \partial D_{r'} \text{, and } h_2:\partial D_1\ra \partial D_1\] 
have $K_2$-q.c.\ extensions to some neighborhoods of these circles, 
then there exists a $K$-q.c.\ homeomorphism \text{$h:A(r)\ra A(r')$} with: 
\begin{itemize}
\item[--] for every $z\in \partial D_r$ we have $h(z)=h_1(z)$, and for every $z\in \partial D_1$ we have $h(z)=h_2(z)$;
\item[--] The curve $h(t)$, for $t\in[1,r]$, has wrapping number \[\theta(h_1(r))-\theta(h_2(1))+2k\pi.\]
\end{itemize}
\end{lem}
If $h_1(r)$ or $h_2(1)$ does not belong to the domain of $\theta$, one may compose $h_1$ and $h_2$ with a small rotation. 
By adding a $+1$ or $-1$ to $k$, the statement still holds independent of the choice of this rotation.
One proves this lemma by explicitly building such homeomorphisms for every $k$. Further details are given in the Appendix.

A homeomorphism $h$ as above is called a \textit{gluing} of $h_1$ and $h_2$ with $k$ twists. If the number of 
the twists is not concerned, we say that $h$ is a gluing of $h_1$ and $h_2$. 

Applying the above lemma to the uniformizations of $\mathbb{V}_{n,0}\setminus \Delta_{n,0}$ and
$\widetilde{\mathbb{V}}_{n,0}\setminus \widetilde{\Delta}_{n,0}$,
with the induced maps from  $h_{n-1,0}$ and $h_{n,0}$ on the 
boundaries, and an integer $k_n$ which will be determined later, gives the \text{$K'$-q.c.\ }homeomorphisms
$g_n$. In the next section we prescribe some special numbers
$k_n$, which are bounded by a constant depending only on $\eps$, in
order to make the $K$-q.c.\ homeomorphism $H$ homotopic
to a topological conjugacy relative $\pc(f)$.

Definite moduli of the annuli $\mathbb{V}_{n,i}\setminus \Delta_{n, i}$ implies that all the nests of the  domains  $\mathbb{V}_{n,i}$ shrink to points in $\pc(f)$. Therefore, $H$ can be extended to a well defined 
\text{$K$-q.c.\ }homeomorphism on $\pc(f)$. 
See~\cite{St} for a detailed proof of this statement. 

\subsection{Isotopy}
Let $\Bpsi_n$ denote the topological conjugacy between $\Bf_{c_n}$ and
$\Bf_{\tilde{c}_n}$ obtained from extending the identity in the B\"ottcher coordinates through $J(\Bf_{c_n})$. The lift $\psi_{n,0}:= \widetilde{S}_n^{-1} \circ \Bpsi_n \circ S_n$ topologically conjugates $\rr^{n} f$ to $\rr^{n} \tilde{f}$ on a neighborhood of $J_{n, 0}$. Note that this neighborhood contains $\Omega_{n, 0}$. 
In the dynamic plane of $\Bf_{c_n}$, let $U(\eta)$ denote the domain enclosed by $E^{\eta}$. 
Define 
\[\psi_{n,i}:=f^{-i} \circ \psi_{n,0} \circ f^{i}:\Delta_{n,i} \ra \mathbb{C},\]
where the inverse branch $f^{-i}$ is chosen so that $\psi_{n,i}(\Delta_{n,i})$ covers a neighborhood of $\pc(f)\cap \tilde{J}_{n,i}$. 
\begin{lem}\label{homos}
Assume that level $n$ belongs to Case $\mathscr{A}$ or $\mathscr{B}$. 
For every $i=0,1,2,\dots, t_n-1$ , $h_{n,i}:\Delta_{n, i} \ra \widetilde{\Delta}_{n, i}$ is homotopic to $\psi_{n,i}:\Delta_{n,i} \ra \tilde{\Delta}_{n,i}$
relative the little Julia sets $J_{n+1,j}$ of level $n+1$ inside $\Delta_{n,i}$.
\end{lem}

\begin{proof}
By the definition of $\Delta_{n,i}$, $\mathbb{V}_{n,i}$, and $h_{n,i}$, it is enough to prove the lemma for $i=0$. For the other ones, one lifts this homotopy via $f^{i}$ and $\tilde{f}^{i}$, or defines them in a similar manner.

Recall that $\Delta_{n,0}$, $\psi_{n,0}$, and $h_{n, 0}$ are the
lifts of $\mathbf{\Delta}_{n, 0}$, $\Bpsi_n$, and $\Bh'_{n, 0}$
under the straightening maps. First we introduce a homotopy between $\Bpsi_n$ and $\Bh'_{n,0}$, relative the Julia sets $\BJ_{1,i}$ of $\Bf_{c_n}$, on the dynamic planes of $\Bf_{c_{n}}$ and $\Bf_{\tilde{c}_{n}}$. Then, we lift
this homotopy to the dynamic planes of $\rr^n f$ and $\rr^n \tilde{f}$ by the straightening maps. Recall
that in the construction, $\Bh'_{n, 0}$ is an adjustment of $\Bh_{n,0}$ through a homotopy 
relative the little Julia sets $\BJ_{1,i}$. Thus, to prove the lemma, we need only show that 
$\Bh_{n,0}$ in homotopic to $\Bpsi_n$ relative the little Julia sets $\BJ_{1,i}$.

First assume that level $n$ belongs to Case $\mathscr{A}$. The idea of the
proof, presented below in detail, is to partition $\BDelta_{n,0}$, by means of rays and
equipotential arcs, into several topological disks and an annulus such
that $\Bpsi_n$ and $\Bh_{n,0}$ are identical on the boundaries of these domains.

Recall the puzzle piece $Q^{\chi_n}_{n,0}$ (where $Q^{\chi_n}_{n, 0}
=Y^{q_{\chi_n}}_{0}$) introduced in this case. 
The equipotential $\Bf_{c_n}^{-\chi_n}(E^\eta)$, and
the rays bounding $Q^{\chi_n}_{n,i}$ up to $\Bf_{c_n}^{-\chi_n}(E^\eta)$, for $i=0, 1, \dots, t_{n}-1$, 
cut the domain $\mathbf{\Delta}_{n, 0}$ into
the annulus $\mathbf{\Delta}_{n, 0} \setminus \Bf^{-\chi_n}(U(\eta))$
and several topological disks. The topological disks which do not intersect the
little Julia sets $\BJ_{1,i}$, the puzzle pieces $Q^{\chi_n}_{n, i}$, and the 
remaining annulus $U(\eta) \setminus f^{-\chi_n}(U(\eta))$ form the appropriate partition.
By Theorem \ref{pseudo-conj}, $\Bh_{n,0}$ and $\Bpsi_n$ are identical on
the boundaries of these domains. 
Indeed, $\Bpsi_n$ is the identity in the B\"{o}ttcher coordinates, and the pseudo-conjugacy
$h_{n, 0}$ obtained in Theorem \ref{pseudo-conj} also matches the B\"{o}ttcher
marking. This proves that the two maps are homotopic outside of the
puzzle pieces $Q^{\chi_n}_{n,i}$ relative $\cup_i \partial Q^{\chi_n}_{n,i}$. 

To define a homotopy inside $Q^{\chi_n}_{n,0}$, recall that we started with a q.c.\ 
homeomorphism $g: Q^{\chi_n}_{n,0} \setminus P^{\chi_n}_{n,0}\to \tilde{Q}^{\chi_n}_{n,0} \setminus \tilde{P}^{\chi_n}_{n,0}$, which was
homotopic to $\Bpsi_n$ relative $\partial (Q^{\chi_n}_{n,0} \setminus P^{\chi_n}_{n,0})$. 
Hence, the lift of $g$ from $A^k_n$ to $\tilde{A}^k_n$, for $k=1,2,3,\dots$, considered in Case $\mathscr{A}$, is homotopic to $\Bpsi_n$ relative $\partial A^k_n$. 
As the two maps are identical on $J_{n+1,0} \subset Q^{\chi_n}_{n,0}$, they glue together to define a homotopy between $g$ and $\Bpsi_n$ on $Q^{\chi_n}_{n,0}$ relative $\partial Q^{\chi_n}_{n,0}\cup J_{n+1,0}$. 

The same argument applies to all other domains $Q^{\chi_n}_{n,i}$ as well.

If level $n$ belongs to Case $\mathscr{B}$, we repeat the above argument on each puzzle piece $Y^0_i$, for $i=0,1,2, \dots, \frac{t_n}{t_{n-1}}-1$.

If level $n$ follows a Case $\mathscr{C}$, one needs to restrict the above homotopies to the smaller domains 
$\BE_{n-1}$ (by the argument at the end of \text{Case $\mathscr{C}$}).
\end{proof}

Assume that level $n$ belongs to Case $\mathscr{A}$ or $\mathscr{B}$, and it follows a Case $\mathscr{A}$ or $\mathscr{B}$. Consider the uniformizations \label{uniforms}
\begin{align*}
\phi_1:A(s) \ra (\mathbb{V}_{n, 0}\setminus J_{n, 0}), \quad &
\phi_2:A(r) \ra (\Delta_{n, 0}\setminus J_{n, 0})\\ 
\widetilde{\phi}_1: A(\tilde{s}) \ra (\mathbb{\widetilde{V}}_{n,0}\setminus \widetilde{J}_{n, 0}),\quad & 
\widetilde{\phi}_2: A(\tilde{r}) \ra (\widetilde{\Delta}_{n, 0}\setminus \widetilde{J}_{n, 0})
\end{align*}
for some constants $s>r$ and $\tilde{s}> \tilde{r}$. The q.c.\ homeomorphisms 
$h_{n-1,0}:\mathbb{V}_{n,0}\setminus J_{n,0}\ra \mathbb{\widetilde{V}}_{n,0}\setminus \widetilde{J}_{n,0}$ 
and $h_{n,0}:\Delta_{n,0}\setminus J_{n,0}\ra \widetilde{\Delta}_{n,0}\setminus \widetilde{J}_{n,0}$ 
lift via $\phi_i$ and $\widetilde{\phi}_i$, $i=1,2$, to q.c.\ homeomorphisms 
$\hat{h}_{n-1, 0}:A(s) \ra A(\tilde{s})$ and $\hat{h}_{n, 0}: A(r) \ra A(\tilde{r})$, respectively, with
the same dilatations. By composing these uniformizations with some rotations, if necessary, we may 
assume that the point one is mapped to the point one under $\hat{h}_{n-1,0}$ and $\hat{h}_{n,0}$. 
Denote the wrapping number of the image of the line segment $[1,s]$ under $\hat{h}_{n-1,0}$
by $\omega_{1,n}$, and denote the wrapping number of the image of the line segment $[1,r]$ under 
$\hat{h}_{n,0}$ by $\omega_{2,n}$.  
By Proposition~\ref{spiral}, the absolute values of $\omega_{1,n}$ and $\omega_{2,n}$ are bounded by some 
constant depending only on $\eps$.
Define $k_n$ as $\omega_{1,n}-\omega_{2,n}$. Let $g_n':D_s\setminus D_r \ra D_{\tilde{s}}\setminus D_{\tilde{r}}$ be a gluing of $\hat{h}_{n,0}: \partial D_r \ra \partial D_{\tilde{r}}$ and
$\hat{h}_{n-1,0}:\partial D_s\ra \partial D_{\tilde{s}}$ with $k_n$ twists, using Lemma~\ref{adjust}. 
With this choice of gluing, the curve $\hat{h}_{n,0}[1,r]\cup g'_n[r,s]$ is homotopic to the curve $\hat{h}_{n-1,0}[1,s]$ inside $D_s\setminus D_1$ relative the boundary points on $\partial (D_s\setminus D_1)$.
Therefore, the map obtained from gluing  $\hat{h}_{n,0}$ to $g_n'$ is homotopic to 
$\hat{h}_{n-1,0}: A(s)\ra A(\tilde{s})$ relative the boundary conditions. 
Let $g_n$ denote the lift of $g_n'$ via $\phi_1$ and $\tilde{\phi}_1$. This homotopy lifts to a homotopy 
between $h_{n-1,0}$ and the map obtained from gluing $g_n$ to $h_{n,0}$.

Before we define $k_n$ in other cases, we need the following extension.
\begin{lem} \label{homos2}
The q.c.\ homeomorphism $\Bh'_{n-1}$ introduced in Case $\mathscr{C}$ admits an 
extension through $\cup_j \BL_{n,j}$ satisfying the following properties
\begin{itemize}
\item[--] For every $j$, $\Bh'_{n-1}:\BL_{n,j}\setminus \BB_{2,j}\ra \widetilde{\BL}_{n,j}\setminus \widetilde{\BB}_{2,j}$ is q.c.\ with uniformly bounded dilatation depending only on $\eps$;
\item[--] $\Bh'_{n-1}=\Bpsi_{n-1}$ on $\cup_j\BB_{2,j}$; 
\item[--] $\Bh'_{n-1}$ is homotopic to $\Bpsi_{n-1}$ relative $\cup_{j} \BB_{2,j}$. 
\end{itemize}
\end{lem}

\begin{proof} First we extend the map through $\BL_{n,0}$.
Consider the fundamental annuli  $S_{n-1}(U_{n,0} \setminus V_{n,0})$ and 
\text{$\widetilde{S}_{n-1}(\widetilde{U}_{n,0} \setminus \widetilde{V}_{n,0})$}
for $\rr\Bf_{c_{n-1}}$ and $\rr\Bf_{\tilde{c}_{n-1}}$. 
Let 
\[g_{n}: S_{n-1}(U_{n,0} \setminus V_{n,0}) \ra
\widetilde{S}_{n-1}(\widetilde{U}_{n,0} \setminus \widetilde{V}_{n,0})\] 
be a q.c.\ homeomorphism with $\rr \Bf_{\tilde{c}_{n-1}} \circ g_n= g_n\circ \rr\Bf_{c_{n-1}}$ on $\partial (S_{n-1}(V_{n,0}))$. 
Lifting $g_n$ onto the preimages of these annuli, via $\rr \Bf_{c_{n-1}}$ and $\rr\Bf_{\tilde{c}_{n-1}}$, 
we obtain a q.c.\ homeomorphism, still denoted by $g_n$, from $S_{n-1}(U_{n,0})\setminus J(\rr^1(\Bf_{c_{n-1}}))$ to
$\tilde{S}_{n-1}(\tilde{U}_{n,0})\setminus J(\rr^1(\Bf_{\tilde{c}_{n-1}}))$. By Lemma~\ref{commute}, $g_n$ (or some
rotation of it) can be extended onto $J(\rr\Bf_{c_{n-1}})=\BJ_{1,0}$ as $\Bpsi_{n-1}$. 
Moreover, these two maps are homotopic relative $\BJ_{1,0}$. 
By a similar argument as in Lemma~\ref{adjusted}, we adjust $g_n$, through a homotopy relative $\BB_{2,0}$, to a q.c.\ homeomorphism mapping $\BL'_{n}$ to $\widetilde{\BL}'_{n}$. 

Consider the three annuli $Y_0^0 \setminus \BL_{n}$, $\BL_{n} \setminus \BL'_{n}$, and $\BL'_{n}\setminus \BB_{2,0}$, as well as the corresponding tilde ones. We have 
\begin{align*}
\Bh'_{n-1}:Y_0^0 \setminus \BL_{n}\ra \widetilde{Y}_0^0 \setminus \widetilde{\BL}_{n},\text{ and}\quad & g_n:\BL'_{n} \setminus \BB_{2,0} \ra  \widetilde{\BL}'_{n} \setminus \widetilde{\BB}_{2,0}.
\end{align*}
To find a gluing of these two maps on the middle annulus $\BL_n\setminus\BL'_n$, we use the above argument to find the right number of twists on this annulus. Consider a curve $\gamma$
connecting a point $a \in\BB_{2,0}$ to a point $d \in \partial Y_0^0$ such that it intersects each of $\partial\BL'_{n}$ and $\partial\BL_{n}$ only once at points denoted by $b$ and $c$, respectively. 
Let us denote by $\gamma_{ab}$, $\gamma_{bc}$, and $\gamma_{cd}$ each segment
of this curve cut off by these four points. The wrapping number $\omega(\Bpsi_{n-1}(\gamma))$ is uniformly bounded depending only on $\mathcal{SL}$ condition. That is because $\Bpsi_{n-1}$ depends continuously on $\tilde{c}_{n-1}$,  and $\tilde{c}_{n-1}$ belongs to a compact set of parameters. 
Therefore, by Proposition~\ref{spiral}, $\omega(\Bpsi_{n-1}(\gamma))-\omega(\Bh_{n-1}(\gamma_{cd}))-\omega(g_n(\gamma_ {ab}))$
is uniformly bounded depending only on $\eps$ and the class $\mathcal{SL}$. 
If we glue $\Bh_{n-1}$ to $g_n$ by such a number of twists (Lemma~\ref{adjust}), the resulting map will be
homotopic to $\Bpsi_{n-1}$ relative $\BB_{2,0}\cup \partial Y_0^0$. 

Similarly, one extends $\Bh'_{n-1}$ onto the other topological disks $\BL_{n,j}$.
\end{proof}

If a Case $\mathscr{C}$ follows a Case $\mathscr{A}$ or a Case $\mathscr{B}$, the number of twists $k_n$
is defined as in the proof of the above lemma. If level $n-1$ belongs to Case $\mathscr{C}$ 
and level $n$ is any of the three cases, we define $k_n$ using the uniformizations of the annuli $\mathbb{V}_{n,0},\setminus B_{n, 0}$, $\Delta_{n,0}\setminus B_{n,0}$, $\widetilde{\mathbb{V}}_{n,0}\setminus \widetilde{B}_{n, 0}$, and $\widetilde{\Delta}_{n,0}\setminus \widetilde{B}_{n,0}$ instead of the uniformizations after the proof of 
Lemma~\ref{homos}.

The following elementary lemma, whose proof appears in the Appendix, is used to put together the homotopies on different levels. 

\begin{lem}\label{interpolation}
Let $U$ and $\widetilde{U}$ be closed annuli with outer boundaries $\gamma_1$ and $\tilde{\gamma}_1$ as well as 
inner boundaries $\gamma_2$ and $\tilde{\gamma}_2$, respectively. Also, let $h_1:\gamma_1 \ra \tilde{\gamma}_1$ 
be a homeomorphism and $h_2^t:\gamma_2 \ra \tilde{\gamma}_2$, for $t \in [0, 1]$, be a continuous family of homeomorphisms. Any continuous gluing $G^0:U\ra \widetilde{U}$ of $h_1$ and $h_2^0$ extends to a continuous family of gluing $G^t:U\ra \widetilde{U} $ of $h_1$ and $h_2^t$, for $t\in[0,1]$.
\end{lem}
{\samepage
\begin{prop}
The q.c.\ homeomorphism $H$, obtained from gluing all the maps $g_{n,i}$ and $h_{n,k}$,
is homotopic to the topological conjugacy $\Psi$ between $f$ and $\tilde{f}$ relative $\pc(f)$.
\end{prop}
}%samepage
\begin{proof}
Let $H_n$ denote the homeomorphism obtained from gluing the maps $g_{i,j}$ and $h_{k,l}$, with $i$ and $k$ less than or equal to $n$.

First, we claim that 
\begin{itemize}
\item[--] the maps $H_1$ and $\Psi$ belong to the same homotopy class of homeomorphisms from $\cc\setminus \cup_i J_{1,i}$ 
to $\cc\setminus \cup_i \widetilde{J}_{1,i}$, and 
\item[--] for every $n>1$, $H_{n-1}$ and $H_n$,  belong to the same homotopy class of homeomorphisms from $\cc\setminus \cup_i {J}_{n+1,i}$ to $\cc\setminus \cup_i \widetilde{J}_{n+1,i}$.
\end{itemize}

By definition, $H_1$ is equal to $h_{1,0}$ which is homotopic to $\psi_{1, 0}$ relative $\cup_i J_{1,i}$, 
by Lemma~\ref{homos} or Lemma~\ref{homos2}. Therefore, $H_1$ is homotopic to $\Psi$, relative $\cup_i J_{1,i}$, 
by Proposition~\ref{commute}. 

Recall that $H_{n-1}$ and $H_n$ are identical on the complement of $\cup_j\mathbb{V}_{n,j}$. 
On $\Delta_{n,i}$, $H_{n-1}$ and $H_n$ are equal to $h_{n-1,i}$ and $h_{n,i}$, respectively.

The domain $\mathbb{V}_{n, 0}$ is partitioned into $\mathbb{V}_{n,0} \setminus \Delta_{n,0}$ and $\Delta_{n,0}$. 
On $\Delta_{n,0}$, the maps $h_{n,0}$ and $h_{n-1,0}$ are homotopic to $\psi_{n,0}$ relative $\cup_i J_{n+1,i}$, either by Lemma~\ref{homos} or \ref{homos2}. 
Thus, there exists a homotopy $h_n^t$, for $t$ in $[0,1]$, which starts with $h_{n,0}$, ends with 
$h_{n-1,0}$, and maps $\partial \Delta_{n,0}$ to $\partial \widetilde{\Delta}_{n,0}$, for all
$t\in [0,1]$. 
At time zero, consider the map $h_{n,0}$ on the
inner boundary of $\mathbb{V}_{n,0}\setminus \Delta_{n,0}$, $h_{n-1,0}$ on the outer boundary
of this annulus, and the gluing $G_n^0:=g_{n,0}$ on the annulus. Applying
Lemma~\ref{interpolation} with $h_{n-1,0}$ on the outer
boundary and $h^t_n$ on the inner boundary, we obtain a continuous
family of gluings $G^t_n:\mathbb{V}_{n,0}\setminus \Delta_{n,0}\ra \widetilde{\mathbb{V}}_{n,0}\setminus \widetilde{\Delta}_{n,0}$ between them. The homeomorphism $G^1_n$ is a gluing of 
$h_{n-1,0}:\partial \mathbb{V}_{n,0} \ra \partial \widetilde{\mathbb{V}}_{n,0}$ and 
$h_{n-1,0}:\partial \Delta_{n,0}\ra \partial \widetilde{\Delta}_{n,0}$. 
But, $G^1_n$ must be homotopic to $h_{n-1,0}$ on $\mathbb{V}_{n,0}\setminus \Delta_{n,0}$ relative the boundaries. That is because these two maps send
a curve joining the two different boundaries to two curves (joining the two boundaries) which are 
homotopic in the annulus $\widetilde{\mathbb{V}}_{n,0}\setminus \widetilde{\Delta}_{n,0}$ relative end points. This comes from our choice of the number of twists for the gluing maps. This completes the proof of the claim.

Let $t_0=0 <t_1<t_2< \cdots$, be an increasing sequence in 
$[0,1]$ converging to $1$. Assume that $H^t$, for $t$ in $[t_0,t_1]$, denotes the homotopy
obtained above between $\Psi$ and $H_1$ relative the little Julia sets $J_{1,i}$. Also, let $H^t$, 
for $t \in [t_n, t_{n+1}]$, $n=1,2,\dots$, denote the homotopy between $H_n$ and
$H_{n+1}$ relative the little Julia sets of level $n+2$.

It follows from the construction that for any fixed $z$, $H^t(z)$ eventually stabilizes  
and equals to $H(z)$. Indeed, the \textit{a priori} bounds assumption implies that the diameters of the topological disks $\mathbb{V}_{n,i}$ tend to zero as $n \ra \infty$. 
Therefore, the supremum distance between $H^t$ and $H$ tends to zero as $t\ra 1$. 
We conclude that $H^t$, for $t$ in $[0,1]$, defines a homotopy between $\Psi$ and $H$ relative $\pc(f)$. 
Hence, $H$ is a Thurston conjugacy between $f$ and $\tilde{f}$.
\end{proof}

\subsection{Promotion to hybrid conjugacy}
\begin{prop} \label{openclosed}
Suppose that all infinitely renormalizable unicritical polynomials in a given combinatorial class $\tau = \{ \M_1, \M_2, \M_3, \dots\}$ satisfying the $\mathcal{SL}$ condition, enjoy \textit{a priori} bounds. Then q.c.\ conjugate maps in this class are hybrid conjugate.
\end{prop}
\begin{proof}
Assume that there are polynomials $P_1$ and $P_2$ in $\tau$ which are q.c.\ equivalent but not hybrid equivalent. Define the set
\begin{align*}
\Omega :=&\{c\in \cc: P_c \text{ is q.c.\ equivalent to } P_1\}\\
        =&\{c\in \cc: P_c \text{ is q.c.\ equivalent to } P_2\}.
\end{align*}
The plan is to show that $\Omega$ is both open and closed in $\cc$, which is not possible.

Theorem~\ref{THM} implies that q.c.\ equivalence is the same as combinatorial equivalence in the class $\tau$. 
Since every combinatorial class is an intersection of a nest of closed sets (connectedness locus copies), $\Omega$ is closed.

Consider a point $P$ in $\Omega$. The polynomial $P$ is not hybrid equivalent to both of $P_1$ and $P_2$ by the  assumption. Assume that it is  not hybrid equivalent to $P_1$ (for the other case just change $P_1$ to $P_2$). 
Let $\phi_1:\cc \ra \cc$, be a $k$-q.c.\ homeomorphism with $\phi_1 \circ P=P_1 \circ \phi_1$. 
Replacing $P_1(z)$ by $\omega^{-1}P_1(\omega z)$, for some $(d-1)$-th root of unity $\omega$, if necessary, 
we may further assume that $\phi_1$ is the identity in the B\"ottcher coordinates. 
Pulling the standard complex structure $\mu_0$ on $\cc$ back under $\phi_1$, we obtain a nontrivial $P$-invariant complex structure $\mu$ on $\cc$ with dilatation bounded by $\frac{k-1}{k+1}$. 
Consider the family of complex structures $\mu_{\lambda}:=\lambda\cdot\mu$, for $\lambda$ in the disk of radius $\frac{k+1}{k-1}$ centered at zero 
$D(0,\frac{k+1}{k-1})\subseteq \mathbb{C}$. 

By the measurable Riemann mapping Theorem \cite{Ah66}, there exists a unique q.c.\ mapping $\phi_{\lambda}:\mathbb{C}\ra \mathbb{C}$, for every $\lambda \in D(0,\frac{k+1}{k-1})$, with $\phi_{\lambda}^*\mu_{\lambda}=\mu_0$, $\phi_\lambda(0)=0$, and $\phi_\lambda(z)=z+O(1)$. 
The map $P_\lambda:=\phi_{\lambda}^{-1}  \circ P \circ \phi_{\lambda}:\mathbb{C}\to \mathbb{C}$, for $\lambda\in D(0,\frac{k+1}{k-1})$, preserves the standard complex structure $\mu_0$. 
By Weyl's Lemma (\cite{Ah66}, Chapter II, Corollary $2$), $P_{\lambda}$  is holomorphic. 
As $P_{\lambda}$ is conjugate to $P$ by $\phi_\lambda$, $P_\lambda$ must be a unicritical polynomial of degree $d$, and $\phi_\lambda $ must map the critical value of $P$ to the critical value of $P_\lambda$.  
At $\lambda=1$ we obtain $P$, and at $\lambda=0$ we obtain $P_1$. 
By the analytic dependence of the solution of the measurable Riemann mapping Theorem on the complex structure, the 
family $P_\lambda$, for $\lambda \in D(0,\frac{k+1}{k-1})$, 
covers a neighborhood of $P$ in $\Omega$. 
That is because the critical value of $P_\lambda$ is equal to the $\phi_\lambda$ of the critical value of $P$, 
and $\phi_\lambda$ depends analytically on $\lambda$. This shows that $P$ is an interior point of $\Omega$, 
and, therefore, $\Omega$ is open. 
\end{proof}
\section{Dynamical description of the combinatorics}
Here we give a detailed dynamical description of the combinatorial classes mentioned in the Introduction. 
%It has been proved in \cite{KL1} that infinitely renormalizable parameters satisfying \textit{decorations} enjoy \textit{a priori} bounds. 
Let $P_c$ be an infinitely renormalizable unicritical polynomial with renormalizations $f_n:=\rr^n(P_c)$, $n=0,1,2,\dots$. Each $f_n$ is hybrid conjugate to a polynomial denoted by $\Bf_{c_n}$. 
Let $Y^1_0(n)$ denote the critical puzzle piece of level $1$ of $\Bf_{c_n}$.  
The dynamical meaning of a parameter $c$ satisfying the \textit{decoration} condition is that there exists a constant $M$ such that for every $n\geq 0$ there are integers $t_n$ and $q_n$, both bounded by $M$, with 
\begin{itemize}
\item[--] $\Bf_n^{kq_n}(0)\in Y^1_0(n)$, for every positive integer $k<t_n$, and
\item[--] $\Bf_n^{t_nq_n}(0)\notin Y^1_0(n)$. 
\end{itemize}
In particular, this condition implies that the number of rays landing at the dividing fixed point of $\Bf_{c_n}$ (denoted by $q_n$ here) is uniformly bounded independent of $n$.  

An infinitely renormalizable parameter is of \textit{bounded type} if the supremum of the relative return times $t_{n+1}/t_{n}$, where $\rr^n(P_c)$ is an appropriate restriction of $P_c^{t_n}$, is bounded. 
By definition, the class of maps satisfying the decoration condition contains the infinitely primitively renormalizable parameters of bounded type.

In Section~\ref{combinatorics}, we associated a sequence of maximal connectedness locus copies 
$\tau(f)=\M^1_d,\M^2_d,\dots$ to every infinitely renormalizable unicritical polynomial-like map $f$. 
Let $\pi_n(\tau(f)):=\M^{n}_d$. 
Define 
\begin{displaymath}
\tau(f,n):=\Bigg\{c\in \M_d\; \Bigg| \; \begin{array}{ll} \text{$P_c(z)=z^d+c$ is at least $n$ times }\\ \text{renormalizable, and }\\ \pi_i(\tau(f))=\pi_i(\tau(P_c)), \text{ for } i=1,2,\dots,n \; 
\end{array}
\Bigg\}.
\end{displaymath}

Given an infinitely renormalizable map $f$ and a sequence of integers $n_0=0< n_1< n_2< \cdots,$ we define the  sequence:
\begin{align*}
(\tilde{\tau}(f),\langle n_i \rangle)&:=\langle\tilde{\M}^{n_1},\tilde{\M}^{n_2},\ldots, \tilde{\M}^{n_k},\ldots\rangle,
\intertext{where}
\tilde{\M}^{n_k}&:=\tau(\rr^{n_{k-1}}f,n_k-n_{k-1}).
\end{align*}

Given a sequence of integers $n_0=0< n_1< n_2< \cdots,$ one can see that there is a one to one correspondence between the two sequences $\tau(f)$ and $(\tilde{\tau}(f),\langle n_i \rangle)$. Thus, one may take the latter one as the definition of the combinatorics of an infinitely renormalizable map.

Consider the main hyperbolic component of $\M_d$. There are infinitely many hyperbolic components of $\M_d$ attached to this component, called primary components. Similarly, there are infinitely many hyperbolic components,  secondary ones, attached to these primary components, and so forth. Consider the set of all hyperbolic components obtained this way, i.e., the ones that can be connected to the main hyperbolic component by a chain of hyperbolic components bifurcating one from another. The closure of this set, plus all possible bounded components of its complement, is called the \textit{molecule} $\mathscr{M}_d$.

An infinitely renormalizable map $f$ is said to satisfy the \textit{molecule condition}, if there exists a constant $\eta>0$ and an increasing sequence of positive integers $n_0=0< n_1< n_2< \cdots,$ such that for all $i\geq 1$
\begin{itemize}
\item[--]$\rr^{n_i}f$ is a primitive renormalization of $\rr^{n_i-1}f$, and 
\item[--]the Euclidean distance between $\tilde{\M}^{n_i}_d$ and $\mathscr{M}_d$ is at least $\eta$.
\end{itemize}

Note that for a map satisfying this condition, there may be infinitely many satellite renormalizable maps in the sequence $\langle \rr^nf \rangle$. However, the condition requires that there are infinitely many primitive levels with the corresponding relative connectedness locus copies uniformly away from $\mathscr{M}_d$. 
By a compactness argument, one can see that the decoration condition is stronger than the molecule condition.

For every $\eps \geq 0$, and every hyperbolic component of $\M_d$, there are at most finitely many limbs attached to this hyperbolic component with diameter bigger than $\eps$ (by the Yoccoz inequality on the size of limbs \cite{H}). This implies that for every $\eta>0$, all but a finite number of the secondary limbs are contained in the $\eta$ neighborhood of $\mathscr{M}_d$. 
This implies that the parameters satisfying the molecule condition also satisfy the $\mathcal{SL}$ condition. Therefore, combining with \cite{KL3} and \cite{KL2} we obtain the corollary stated in the Introduction.
\appendix
\section*{Appendix}
\begin{proof}[Proof of Proposition \ref{commute}]
Consider an external ray $R$ landing at a non-dividing fixed point $\beta_0$ of $f$. As $R$ is invariant under $f$, and $\phi$ commutes with $f$, $\phi(R)$ is also invariant under $f$. This implies that $\phi(R)$ lands at a non-dividing fixed point $\beta_j$ of $f$. Choose $\rho_j$ such that $\rho_j(\phi(R))$ lands at $\beta_0$. Let $\psi$ denote the map $\rho_j \circ \phi$, and $R'$ denote the ray $\rho_j(\phi(R))$. For such a rotation $\rho_j$, $\psi$ also commutes with $f$, and $R'$ is also invariant under $f$.

The external ray $R$ cuts the annulus $V_1 \setminus V_2$ into a quadrilateral $I_{0,1}$. The preimage 
$f^{-1}(I_{0,1})$ consists of $d$ quadrilaterals denoted by $I_{1,1},I_{1,2}, \ldots, I_{1,d},$ 
ordered clockwise starting with $R$. Similarly, the $f^{-n}(I_{0,1})$ produces $d^n$ quadrilaterals $I_{n,1},I_{n,2}, \ldots, I_{n,d^n}$ (ordered similarly). In the same way, the external ray $R'$ produces quadrilaterals denoted by $I'_{n,j}$, ordered clockwise starting with $R'$. We claim that the Euclidean diameter of $I_{n,j}$ (and $I'_{n,j}$) goes to zero as $n$ tends to infinity.

Denote $f^{-i}(V_1)$ by $V_{i+1}$, and let $d_{i+1}$ denote the hyperbolic metric on the annulus $V_{i+1}\setminus K(f)$. As $I_{n,j} \subseteq V_n$, and $\cap_n V_n=K(f)$, the quadrilaterals $I_{n,j}$ converge to the boundary of $V_1 \setminus K(f)$ as $n$ goes to infinity. To show that the Euclidean diameters of these quadrilaterals go to zero, it is enough to show that their hyperbolic diameters in $(V_1 \setminus K(f),d_1)$ stay bounded from above. 
Since $f^{n-1}:(V_n \setminus K(f),d_n) \ra (V_1 \setminus K(f),d_1)$ is an unbranched covering of degree $d^{n-1}$, it is a local isometry. As the closure of $f^{n-1}(I_{n,j})$ is a compact subset of $V_1 \setminus K(f)$, we conclude that $I_{n,j}$ has bounded hyperbolic diameter in $(V_n\setminus K(f),d_n)$. Finally, the contraction of the inclusion from \text{$(V_{n} \setminus K(f),d_n)$} into $(V_1 \setminus K(f),d_1)$ implies that $I_{n,j}$ has bounded hyperbolic diameter in $(V_1\setminus K(f),d_1)$. 

By a similar argument, one can show that the hyperbolic distance between $I_{n,j}$ and $I'_{n,j}$ in 
$(V_1 \setminus K(f),d_1)$ is uniformly bounded from above.

Since $\psi$ is a conjugacy, it sends $I_{n,j}$ to $I'_{n,j}$. Therefore, as $w$ converges to $K(f)$, $w$ and $\psi(w)$ belong to some $I_{n,j}$ and $I'_{n,j}$, respectively, with larger and larger values of $n$. 
Combining with the above argument, we conclude that the Euclidean distance between $w$ and $\psi(w)$ tends to zero as $w\ra K(f)$. 
This implies that $\psi$ extends through $K(f)$ as the identity.
\end{proof}

\begin{comment}
\begin{proof}[proof of lemma \ref{spiral}]
Consider the curve family $L$ consisting of all ray segments in
$A(r)$ obtained from rotating the real line segment $[1,r]$ about the origin,
that is, the radial lines in $A(r)$.

We denote by $\Gamma(F)$ the extremal length of a given curve
family $F$. The inequality $\Gamma(L)/K \leq
\Gamma(\psi(L)) \leq K \Gamma(L)$ holds for the $K$-q.c.\ map $\psi$. See~\cite{Ah66} for more details on
curve families and extremal length properties.

It is easy to see that if the interval $[1,r]$ is mapped to a curve
with wrapping number $T$, then every curve in $L$ is mapped to a
curve with wrapping number between $T+1$ and $T-1$.

By definition of extremal length and choosing conformal metric $\rho$ as the Euclidean metric, we obtain
\[K \Gamma(L) \geq  \Gamma (\psi(L)) = \sup_{\rho} \frac{\inf \ell_{\rho}(\psi(\gamma))^2}
{S_{\rho}} \geq \frac{4 \pi(T-1)^2}{R'^2-1},\]
and $\Gamma(L)= \log R/2\pi.$
By properties of q.c.\  mappings, we have $R' \leq R^K$ which implies
\[T \leq \frac{1}{\pi}  \sqrt{\frac{K \log(R) ( R^{2K} -1)}{8}}+1 \]
\end{proof}
\end{comment}

\begin{proof}[Proof of Lemma \ref{adjust}]
Let $\Pi_{r}:=\{z\mid 0< \Im(z)< \frac{1}{2\pi}\log r\}$, for $r>1$, denote the covering space of $A(r)$ with the deck transformation group generated by $z\ra z+1$. 
We may assume that $\log r$ and $\log r'$ are at least $6\pi$. Otherwise, one may rescale these strips by affine maps of the form $(x,y)\mapsto (x,ay)$ and continue with the following argument. 
 
The homeomorphisms $h_1$ and $h_2$ lift to $1$-periodic homeomorphisms 
\begin{align*}
&\hat{h}_1: \mathbb{R}+\frac{\Bi}{2\pi}\log r \ra \mathbb{R}+\frac{\Bi}{2\pi}\log r', \text{ and }
\hat{h}_2:\mathbb{R}\ra \mathbb{R}, \\
&\text{with }\hat{h}_2(0)=0 \text{ , and }\\
&\hat{h}_1(\frac{\Bi}{2\pi}\log r)=\theta(h_1(r))-\theta(h_2(1))+2\pi k+\frac{\Bi}{2\pi} \log r'. 
\end{align*}
As these maps have $K_2$-q.c.\ extensions to some neighborhoods of their domains of definition, 
they are quasi-symmetric with 
a constant $M(K_2)$ (See Theorem $1$ in \cite{Ah66}, Page $40$).   

To prove the lemma, it is enough to introduce a $1$-periodic q.c.\ mapping $h:\Pi_r\ra \Pi_{r'}$, matching $\hat{h}_1$ and $\hat{h}_2$ on the boundaries, and with uniformly bounded dilatation in terms of $K_1,K_2,k$, and $\delta$.  

Define $\phi: \Pi_\infty \ra \Pi_\infty$ as $\phi(x,y):= u(x,y)+ \Bi v(x,y)$, where
\begin{align*}
u(x,y):=& \frac{1}{2y}\int_{-y}^{+y} \hat{h}_2(x+t) \,dt,\\
v(x,y):=& \frac{1}{2y}\int_{0}^{y} (\hat{h}_2(x+t)-\hat{h}_2(x-t)) \,dt.
\end{align*}
It has been proved in \cite{Ah66}, Page $42$, that $\phi$ is a q.c.\ mapping with dilatation depending only 
on $M(K_2)$. Note that $\phi$ is $1$-periodic in the first variable. 
It follows from Lemma $3$ on Page $41$ of \cite{Ah66} that 
\[v(0,1)\leq \int_0^1 \hat{h}_2(t)\,dt-\frac{1}{2}\leq \frac{M(K_2)}{M(K_2)+1}\leq \frac{1}{2}\in [0,3].\]
Also,
\[v_x(x,1)=0, \text{ and } u_x(x,1)=1, \text{ for }-\infty <x<\infty.\]
This implies that $\phi$ maps the horizontal line through $\Bi$ to a horizontal line in $\Pi_{r'}$ as a  translation.
Similarly, one extends $\hat{h}_1$ to a q.c.\ homeomorphism $\psi$ of $\{z\mid \Im(z)\leq \frac{1}{2\pi}\log r\}$ such that it maps the horizontal line through $\frac{\Bi}{2\pi} \log r -1$ to a horizontal line in $\Pi_{r'}$ as a translation.    

Consider the map $H:\{z\mid 1\leq  \Im(z)\leq \frac{1}{2\pi}\log r -1\}\ra \Pi_{r'}$ defined by 
\[(x,y)\mapsto (1-y)\phi(x,1)+y \psi(x, \frac{1}{2\pi}\log r -1).\] 
The homeomorphism $H$ is affine with dilatation depending only on $k$. The homeomorphisms $\phi$ and $\psi$ are 
q.c.\ with dilatations depending only on $M(K_2)$. The homeomorphism 
\[\begin{cases} \phi(x,y) &\!\!\text{if $0 \leq x \leq 1$}\\
                H(x,y)    &\!\!\text{if $  1 \leq x \leq \frac{1}{2\pi}\log r-1$} \\
                \psi(x,y) &\!\!\text{if $  \frac{1}{2\pi}\log r-1 \leq x \leq \frac{1}{2\pi}\log r$}
   \end{cases}\]
is the desired q.c.\ mapping, with dilatation depending only on $k, K_1, \delta$, and $M(K_2)$.
\end{proof}

\begin{proof}[Proof of Lemma \ref{interpolation}]
As this is a purely topological statement, we may assume that $U$ and $\widetilde{U}$ are $D_2\setminus D_1$ whose   universal covering space is $\Pi:=\mathbb{R}\times (0,1)$.  
Lifting $h_1$, $h_2^t$, and $G^0$ via the projections from the closure of $\Pi$ to the closures of $U$ and $\widetilde{U}$, we obtain homeomorphisms $\hat{h}_1:\mathbb{R}\times \{0\}\ra \mathbb{R}\times \{0\} $, $\hat{h}_2^t:\mathbb{R}\times \{1\}\ra \mathbb{R}\times \{1\} $, and $\hat{G}^0:\Pi\ra \Pi$, all $1$-periodic in the first coordinate. 
Composing all the above maps with a rotation if necessary, we may assume that $\hat{h}_1(0,0)=(0,0)$. 
The lift $\hat{G}^0$ is uniquely determined by being an extension of $\hat{h}_1$, and then $\hat{h}_2$ is a unique 
extension of $\hat{G}^0$.
Define the continuous family of $1$-periodic homeomorphisms $T^t: \Pi\ra\Pi$, for $t\in [0,1]$, as
\[T^t(x,y):= (1-y)\cdot \hat{h}_1(x,0)+y \cdot \hat{h}_2^t(x,1).\]
By defining the $1$-periodic homeomorphism $H: \Pi\ra\Pi$ as $(\hat{G}^0)^{-1}  \circ T^0$, one can see that
$T^t \circ H^{-1}$ is a continuous family of gluings of $\hat{h}_1$ and $\hat{h}_2^t$ starting with $\hat{G}^0$. 
This periodic family projects to a continuous family of interpolations of $\hat{h}_1$ and $\hat{h}_2^t$ on $U$. 
\end{proof}
\bibliographystyle{amsalpha}
\bibliography{Data}
\end{document}